\documentclass[a4paper,11pt]{article}
\usepackage[utf8]{inputenc}

\usepackage[top=2.5cm,bottom=2.5cm,left=2.5cm,right=2.5cm]{geometry}
\usepackage[english]{babel}
\usepackage{tikz}
\usepackage{devanagari}
\usepackage{fancyvrb}

\title{
Small Weight Code Words of Projective Geometric Codes}
\author{Sam Adriaensen \\  {\it Vrije Universiteit Brussel} \and Lins Denaux \\ {\it Ghent University}}
\date{}

\usepackage{amsmath, amssymb, amsthm}
\usepackage{xcolor, comment, url}

\makeatletter
\newtheorem*{rep@theorem}{\rep@title}
\newcommand{\newreptheorem}[2]{%
\newenvironment{rep#1}[1]{%
 \def\rep@title{#2 \ref{##1}}%
 \begin{rep@theorem}}%
 {\end{rep@theorem}}}
\newtheoremstyle{indentedObs}
  {3pt}
  {3pt}
  {\addtolength{\@totalleftmargin}{2.5em}
   \addtolength{\linewidth}{-2.5em}
   \parshape 1 2.5em \linewidth}
  {}
  {\itshape}
  {.}
  {.5em}
  {}
\makeatother

\newtheorem{thm}{Theorem}[section]
\newreptheorem{thm}{Theorem}
\newtheorem{lm}[thm]{Lemma}
\newtheorem{res}[thm]{Result}
\newtheorem{crl}[thm]{Corollary}
\newreptheorem{crl}{Corollary}

\theoremstyle{definition}
\newtheorem{rmk}[thm]{Remark}
\newtheorem{df}[thm]{Definition}
\newtheorem{constr}[thm]{Construction}

\theoremstyle{indentedObs}
\newtheorem{obs}{Observation}

\newcommand{\mc}{\mathcal{C}}
\newcommand{\ms}{\mathcal{S}}
\newcommand{\mh}{\mathcal{H}}
\newcommand{\pr}[3]{\textnormal{proj}_{#1,#2}(#3)}
\newcommand{\prj}[4]{\textnormal{proj}^{(#1)}_{#2,#3}(#4)}
\newcommand{\wt}{\textnormal{wt}}
\newcommand{\supp}{\textnormal{supp}}
\newcommand{\pa}{\textnormal{\dn p}}
\newcommand{\la}{\textnormal{\dn l}}
\newcommand{\gauss}[2]{\genfrac{[}{]}{0pt}{}{#1}{#2}_q}
\newcommand{\one}{\mathbf 1}
\newcommand{\zero}{\mathbf 0}
\newcommand{\vspan}[1]{\left \langle #1 \right \rangle}
\newcommand{\fp}{\mathbb F_p}
\newcommand{\set}[1]{\left \{ #1 \right \}}
\newcommand{\sett}[2]{\left \{ #1 \, : \, #2 \right \}}
\newcommand{\pg}{\textnormal{PG}}
\newcommand{\eps}{\varepsilon}

\renewcommand{\geq}{\geqslant}
\renewcommand{\leq}{\leqslant}

\VerbatimFootnotes

\begin{document}

\maketitle

\begin{abstract}
    We investigate small weight code words of the $p$-ary linear code $\mc_{j,k}(n,q)$ generated by the incidence matrix of $k$-spaces and $j$-spaces of $\pg(n,q)$ and its dual, with $q$ a prime power and $0 \leq j < k < n$.
    Firstly, we prove that all code words of $\mc_{j,k}(n,q)$ up to weight $\left(3 - \mathcal{O}\left(\frac 1 q \right) \right) \gauss{k+1}{j+1}$ are linear combinations of at most two $k$-spaces (i.e.\ two rows of the incidence matrix).
    As for the dual code $\mc_{j,k}(n,q)^\perp$, we manage to reduce both problems of determining its minimum weight (1) and characterising its minimum weight code words (2) to the case $\mc_{0,1}(n,q)^\perp$.
    This implies the solution to both problem (1) and (2) if $q$ is prime and the solution to problem (1) if $q$ is even.
\end{abstract}



\section{Introduction}

To keep things clear and compact, we will postpone introducing the necessary preliminaries; see Section \ref{SecPrelim} for an overview of all notations and known results used throughout this article.

A main research topic in coding theory is finding the minimum weight of certain linear codes and characterising its minimum weight code words (or, more generally, code words of a relatively small weight).
This article investigates small weight code words of $\mc_{j,k}(n,q)$ and $\mc_{j,k}(n,q)^\perp$, which are the $p$-ary linear code generated by the incidence matrix of $k$-spaces and $j$-spaces of $\pg(n,q)$ and its dual, respectively.\\

Some important characterisations are already known.
Namely, the minimum weight of $\mc_{j,k}(n,q)$ is equal to the number of $j$-spaces in a $k$-space, and code words corresponding to this weight are characterised as being scalar multiples of $k$-spaces (Result \ref{ResLowestWeight}).
Moreover, narrowing our view to the code $\mc_{0,k}(k+1,q)$, all code words of weight at most $\left(3 - \mathcal{O}\left(\frac 1 q \right) \right) q^k$ are characterised as being linear combinations of at most two $k$-spaces (Result \ref{ResLins}).
Subtly brushing away the fact that the authors of the latter proved a slightly stronger result, no other results are known concerning small weight code words of $\mc_{j,k}(n,q)$ if $n > 2$.

Less is known about the dual code $\mc_{j,k}(n,q)^\perp$.
In general, the minimum weight of $\mc_{j,k}(n,q)^\perp$ is not known.
However, this minimum weight has upper bound $2q^{n-k}$.
If $q$ is prime, the minimum weight of $\mc_{j,j+1}(n,q)^\perp$ is equal to this bound and its minimum weight code words are characterised as being scalar multiples of so-called \emph{standard words} (Definition \ref{DefStandardWord}, Result \ref{ResMaxMinWtDualCode}).
If $q$ is even, the minimum weight of $\mc_{0,k}(n,q)^\perp$ equals $(q+2)q^{n-k-1}$ (Result \ref{ResMaxMinWtDualCodeQeven}).\\

A further overview of results on these codes can be found in \cite{lavrauwC} and \cite{lins}.

\section{Outline and main results}

As mentioned before, all preliminaries needed to guide you through this article can be found in Section \ref{SecPrelim}.

In section \ref{SecRelDualCode}, we study the relation between $\mc_{j,k}(n,q)$, $\mc_{j,n-k+j}(n,q)^\perp$, their intersection (i.e. the \emph{hull} $\mh_{j,k}(n,q)$ of $\mc_{j,k}(n,q)$) and their span.
We bundle several properties that were already known for specific values of $j$, $k$, $n$ and $q$, and present them in a general context.

In Section \ref{SectPointsKSpaces} and Section \ref{SectJKSpaces}, we investigate the small weight code words of $\mc_{0,k}(n,q)$ and $\mc_{j,k}(n,q)$, respectively.
In Section \ref{SectPointsKSpaces}, we use the known results concerning small weight code words of $\mc_{0,k}(k+1,q)$ to characterise all code words of $\mc_{0,k}(n,q)$ up till weight $W(k,q)$.
The exact value of the latter bound (as well as the meaning of the sets $Q_i$) can be found in Definition \ref{DefW}, but for the sake of simplicity, one can view this bound to be roughly equal to $(3-3/q)q^k$ if $q$ is large enough.
\begin{repthm}{ThmPointsKSpaces}
If $c$ is a code word of $\mc_k(n,q)$, with $\wt(c) \leq W(k,q)$, then $c$ is a linear combination of at most two $k$-spaces.
Moreover, if $q \in Q_3 \cup Q_4 \cup Q_5$, then this bound is tight.
\end{repthm}
In particular, the minimum weight code words of the hull $\mh_{0,k}(n,q)$ are characterised as well.
\begin{repcrl}{CrlMinWtHullPoints}
If $c$ is a code word of $\mh_{0,k}(n,q)$, with $\wt(c) \leq W(k,q)$, then $c$ is a scalar multiple of the difference of two $k$-spaces. In particular, the minimum weight of $\mh_{0,k}(n,q)$ is $2q^k$, and the minimum weight code words are scalar multiples of the difference of two $k$-spaces through a common $(k-1)$-subspace.
\end{repcrl}
These results, in turn, are used in Section \ref{SectJKSpaces} as base cases to characterise all code words of $\mc_{j,k}(n,q)$ and $\mh_{j,k}(n,q)$ up till weight $W(j,k,q)$.
Again, the exact value of the latter bound can be found in Definition \ref{DefWj}, but it is at least $(3-7/q) \gauss{k+1}{j+1}$ if $q$ is large enough.
\begin{repthm}{ThmLargeQMainThm}
Assume that $q \notin Q_1$.
\begin{enumerate}
 \item[(1)] If $c$ is a code word of $\mc_{j,k}(n,q)$, with $\wt(c) \leq W(j,k,q)$, then $c$ is a linear combination of at most two $k$-spaces.
 \item[(2)] If $c$ is a code word of $\mh_{j,k}(n,q)$, with $\wt(c) \leq W(j,k,q)$, then $c$ is a scalar multiple of the difference of two $k$-spaces.
 In particular, the minimum weight of $\mh_{j,k}(n,q)$ is $2q^{k-j} \gauss{k}{j}$, and the minimum weight code words are scalar multiples of the difference of two $k$-spaces through a common $(k-1)$-space.
\end{enumerate}
\end{repthm}
The following, somewhat weaker result is valid for any prime power $q$.
\begin{repthm}{ThmMainThmSmallQ}
 If $c$ is a code word of $\mc_{j,k}(n,q)$, with
 \[
 \wt(c) \leq \frac{2q^k}{\theta_j} \gauss k j,
 \]
 then $c$ is a scalar multiple of a $k$-space.
 As a consequence, the minimum weight of $\mh_{j,k}(n,q)$ is larger than $2q^k \gauss k j / \theta_j$.
\end{repthm}
As a final note to this chapter, we investigate the cyclicity of $\mc_{j,k}(n,q)$.
\begin{repthm}{ThmCyclic}
The code $\mc_{j,k}(n,q)$ is equivalent to a cyclic code if and only if $j = 0$.
\end{repthm}

In Section \ref{SectDual}, we shift our focus to the dual code $\mc_{j,k}(n,q)^\perp$ and manage to reduce both problems of determining its minimum weight and characterising its minimum weight code words to the codes $\mc_{0,1}(n,q)^\perp$.
This is done using the construction of a \emph{pull-back} (Construction \ref{ConstrPullBack}).
Pull-backs are code words of $\mc_{j,k}(n,q)^\perp$ constructed from code words of $\mc_{0,k-j}(n-j,q)^\perp$.
\begin{repthm}{ThmPullBack}
    If $j>0$, then all minimum weight code words of $\mc_{j,k}(n,q)^\perp$ are pull-backs.
\end{repthm}
As a consequence, known results concerning $\mc_{j,k}(n,q)^\perp$ are found to be valid for general $j$ and $k$.
\begin{repcrl}{CorGeneralDual}
\begin{enumerate}
 \item[(1)] $d \left( \mc_{j,k}(n,q)^\perp \right) = d \left( \mc_{0,1}(n-k+1,q)^\perp \right)$.
 \item[(2)] If $p$ is prime, the minimum weight code words of $\mc_{j,k}(n,p)^\perp$ are scalar multiples of the standard words, and thus have weight $2p^{n-k}$.
 \item[(3)] If $q$ is even, then $d \left ( \mc_{j,k}(n,q)^\perp \right) = (q+2)q^{n-k-1}$.
\end{enumerate}
\end{repcrl}

We conclude this article with Section \ref{SecOpenProblem} by briefly discussing some open problems concerning these codes.

\section{Preliminaries}\label{SecPrelim}

\subsection{Basic notation}

Throughout this entire article, we will assume $p$ to be a prime number and $q:=p^h$, with $h \in \mathbb{N}^*$.
Moreover, we consider natural numbers $j$, $k$ and $n$, with the general assumption that
\[
    0 \leq j < k < n\textnormal{.}
\]
As a consequence, we will sporadically use the fact that $k \geq 1$ and $n \geq 2$.

We will denote the Desarguesian projective space of (projective) dimension $n$ over $\mathbb F_q$ by $\pg(n,q)$.
For any number $m \in \mathbb{N}$, the number of $j$-spaces in $\pg(m,q)$ is given by the Gaussian coefficient
\[
\gauss{m+1}{j+1} := \frac{ (q^{m+1}-1)(q^m-1) \cdots (q^{m-j+1}-1) }{ (q^{j+1}-1)(q^j-1) \cdots (q-1) }.
\]
By convention, we define $\gauss{m+1}{0}$ to be $1$ and we denote $\theta_m:=\gauss{m+1}{1}$, with the extension that $\theta_m:=0$ for values $m\in\mathbb{Z}\setminus\mathbb{N}$.

Denote the set of all $j$-subspaces of a projective space $\pi$ by $G_j(\pi)$.
We denote the latter by $G_j(n,q)$ if $\pi$ is the ambient space $\pg(n,q)$.
If $\pi$ or $n$ and $q$ are clear from context, we will denote this simply by $G_j$.
Let $V(j,\pi)$ denote the $p$-ary vector space of functions from $G_j(\pi)$ to $\fp$, i.e.\ $V(j,\pi) := \fp^{G_j(\pi)}$.
Similarly, $V(j,n,q) := \fp^{G_j(n,q)}$.
We will denote the functions that map everything to one, respectively zero, by $\one$, respectively $\zero$.

We can identify a $k$-space $\kappa$ of $\pg(n,q)$ with the function $\kappa^{(j)} \in V(j,n,q)$ such that
\[
 \kappa^{(j)}(\lambda) = 
 \begin{cases}
 1 & \text{if } \lambda \subseteq \kappa,\\
 0 & \text{otherwise.}
 \end{cases}
\]

If $j$ is clear from context, we will denote $\kappa^{(j)}$ as $\kappa$.
There should be no confusion.
Let $\mc_{j,k}(n,q)$ denote the subspace of $V(j,n,q)$, generated by $G_k(n,q)^{(j)} := \sett{ \kappa^{(j)}}{\kappa \in G_k(n,q)}$.
We will also denote $\mc_{0,k}(n,q)$ as $\mc_k(n,q)$.

Alternatively, one could define the code $\mc_{j,k}(n,q)$ as follows.
Consider the $p$-ary incidence matrix $A$ of $k$-spaces and $j$-spaces, i.e.\ the rows of the matrix correspond to the $k$-spaces of $\pg(n,q)$ and the columns to the $j$-spaces.
Put a one in the matrix if the $j$-space corresponding to the column is completely contained in the $k$-space corresponding to the row, and zero otherwise.
Symbolically,
\begin{align*}
 A \in \fp^{G_k \times G_j} 
 && \text{and} &&
 A_{\kappa, \lambda} =
 \begin{cases}
 1 & \text{if } \lambda \subseteq \kappa,\\
 0 & \text{otherwise.}
 \end{cases}
\end{align*}
In this way, $\mc_{j,k}(n,q)$ is the row span of the matrix $A$.
However, we prefer the definition of $\mc_{j,k}(n,q)$ as a vector subspace of $V(j,n,q)$, as this is more convenient for notation.

If $v \in V(j,n,q)$, define the \emph{support} of $v$ as $\supp(v) := \set{\lambda \in G_j:v(\lambda) \neq 0}$ and the \emph{weight} of $v$ as $\wt(v) := |\supp(v)|$.
For a vector subspace $W$ of $V(j,n,q)$, let $d(W)$ denote the minimum weight of $W$, i.e.\ $d(W) := \min \set{ \wt(c) : c \in W \setminus \set \zero}$.
For $0 \leq i < j$, we will also make use of the set $\supp_i(c) := \set{\iota \in G_i : (\exists \lambda \in \supp(c))(\iota \subset \lambda)} = \bigcup_{\lambda \in \supp(c)} G_i(\lambda)$.

Define the \emph{scalar product} of two functions $v,w \in V(j,n,q)$ as
\[
 v \cdot w := \sum_{\lambda \in G_j} v(\lambda) w(\lambda).
\]
Define the \emph{dual code} of $\mc_{j,k}(n,q)$ as its orthogonal complement with respect to the above scalar product.
This means that the dual code is
\[
 \mc_{j,k}(n,q)^\perp := \set{v \in V(j,n,q) : (\forall c \in \mc_{j,k}(n,q))(c \cdot v = 0)}.
\]
Define the \emph{hull} $\mh_{j,k}(n,q)$ of $\mc_{j,k}(n,q)$ as
\[
    \mh_{j,k}(n,q) := \mc_{j,k}(n,q) \cap \mc_{j,n-k+j}(n,q)^\perp.
\]

\subsection{Known results and the bounds $\boldsymbol{W(k,q)}$ and $\boldsymbol{W(j,k,q)}$}

Some important characterisations are already known.

\begin{res}[{\cite[Theorem 1]{bagchi}}]\label{ResLowestWeight}
The minimum weight of $\mc_{j,k}(n,q)$ is $\gauss{k+1}{j+1}$, and minimum weight code words are scalar multiples of $k$-spaces, i.e.\ scalar multiples of the elements of $G_k(n,q)^{(j)}$.
\end{res}

If $j = 0$, stronger characterisations are known.

\begin{df}
\label{DefW}
Define $W(k,q)$ as
\[
 W(k,q) := 
 \begin{cases}
  2q^k & \text{if } q \in Q_1 := \sett q { q \leq 9} \cup  \set{16,25,27,49},\\
  2\theta_k & \text{if } q \in Q_2 := \sett q {9 < q \leq 23} \cup \set{29,31,32,121},\\
  3q^k-3q^{k-1}-1 & \text{if } q \in Q_3 := \sett q {q > 32, \, q \text{ prime}},\\
  3q^k-3q^{k-1}+\theta_{k-2}-1 & \text{if } q \in Q_4 := \sett q {q > 32, \, q \text{ even}},\\
  3q^k-2q^{k-1}+\theta_{k-2}-1 & \text{if } q \in Q_5, \text{ the complement of } \bigcup_{i=1}^4 Q_i.
  \end{cases}
\]
\end{df}

We will use the following weakened version of known characterisations.

\begin{res}[{\cite[Corollary 2.2.13]{lins} \cite[Theorem 1.4]{pol}}]\label{ResLins}
If $c$ is a code word of $\mc_k(k+1,q)$, with $\wt(c) \leq W(k,q)$, then $c$ is a linear combination of at most two $k$-spaces.
Moreover, this bound is tight if $q \in Q_3 \cup Q_4 \cup Q_5$.
\end{res}

In Section \ref{SectPointsKSpaces} we prove that this holds for all codes $\mc_k(n,q)$.


\begin{df}
\label{DefWj}
Define $W(j,k,q)$ as
\[
 W(j,k,q) := 
 \begin{cases}
  \frac{2 q^k}{\theta_j} \gauss{k}{j} & \text{if } q \in Q_1,\\
  2 \gauss{k+1}{j+1} & \text{if } q \in Q_2,\\
  \left(3 - \frac 7 q \right) \gauss{k+1}{j+1} & \text{if } q \in Q_3 \cup Q_4,\\
  \left(3 - \frac 6 q \right) \gauss{k+1}{j+1} & \text{if } q \in Q_5.
  \end{cases}
\]
\end{df}

Remark that $W(0,k,q) \leq W(k,q)$.
The focus of Section \ref{SectJKSpaces} are Theorems \ref{ThmLargeQMainThm} and \ref{ThmMainThmSmallQ}, where we prove that code words of $\mc_{j,k}(n,q)$ up to weight $W(j,k,q)$ are linear combinations of at most two $k$-spaces.

\begin{df}
 \label{DefStandardWord}
Let $\iota$ be a $(j-1)$-space, and let $\pi$ and $\rho$ be two $(n-k+j)$-spaces through an $(n-k+j-1)$-space containing $\iota$.
Define $v \in V(j,n,q)$ as
\[
 v :=
 \sum_{\substack{\lambda \in G_j(\pi) \\ \iota \subset \lambda}} \lambda^{(j)} - \sum_{\substack{\lambda' \in G_j(\pi) \\ \iota \subset \lambda'}} \lambda'^{(j)}.
\]
Code words of this form are called \emph{standard words} of $\mc_{j,k}(n,q)^\perp$.
\end{df}

\begin{res}[{\cite[Theorem 3, Proposition 2]{bagchi}}]\label{ResMaxMinWtDualCode}
 Standard words of $\mc_{j,k}(n,q) ^\perp$ are code words of $\mc_{j,k}(n,q) ^\perp$ of weight $2q^{n-k}$.
 Therefore, the minimum weight of $\mc_{j,k}(n,q)^\perp$ is at most $2q^{n-k}$.
 Moreover, if $p$ is prime, then the minimum weight code words of $\mc_{j,j+1}(n,p)^\perp$ are the scalar multiples of the standard words.
\end{res}

\begin{res}[{\cite[Theorem 1]{calkin}}]\label{ResMaxMinWtDualCodeQeven}
 If $q$ is even, then $d \left ( \mc_k(n,q)^\perp \right) = (q+2)q^{n-k-1}$.
\end{res}

\section{A brief note on the relation with the dual code}\label{SecRelDualCode}

As a generalisation of \cite[Chapter 6]{assmus} and \cite[Lemma 2]{lavrauwB}, we have the following.

\begin{lm}
\label{LmHull}
\begin{enumerate}
    \item[(1)] If $c \in \mc_{j,k}(n,q)$, then $c \cdot \pi$ is equal for all subspaces $\pi$ in $\pg(n,q)$ with $\dim(\pi) \geq n-k+j$.
    \item[(2)] $\mh_{j,k}(n,q) = \set{c \in \mc_{j,k}(n,q): c\cdot \one = 0}
    = \vspan{\kappa - \kappa' : \kappa \in G_k}$ for any $\kappa' \in G_k$.
    \item[(3)] $\dim\big(\mh_{j,k}(n,q)\big)  = \dim\big(\mc_{j,k}(n,q)\big) - 1$.
\end{enumerate}
\end{lm}

\begin{proof}
(1) Take a $k$-space $\kappa$ and a subspace $\pi$ with $\dim(\pi) \geq n-k+j$.
It is easy to see that, when considered as elements of $V(j,n,q)$, $\kappa \cdot \pi$ equals the number of $j$-spaces in $\kappa \cap \pi$ modulo $p$.
By Grassmann's identity, $\dim(\kappa \cap \pi) \geq \dim(\kappa) + \dim(\pi) - n \geq j$.
Therefore, the number of $j$-spaces in $\kappa \cap \pi$ equals $\gauss {\dim(\kappa \cap \pi)+1}{j+1} \equiv 1 \pmod p$.
Now take a code word $c \in \mc_{j,k}(n,q)$.
Then $c$ is a linear combination of $k$-spaces, so $c = \sum_i \alpha_i \kappa_i$ for some $\alpha_i \in \fp$ and $\kappa_i \in G_k$.
Since the scalar product is linear, we have that
\[
 c \cdot \pi = \Big(\sum_i \alpha_i \kappa_i\Big) \cdot \pi
 = \sum_i \alpha_i (\kappa_i \cdot \pi)
 = \sum_i \alpha_i,
\]
hence $c \cdot \pi$ is equal for all $\pi$.\\

(2, 3) Take a code word $c \in \mc_{j,k}(n,q)$.
Then $c \in \mc_{j,n-k+j}(n,q)^\perp$ if and only if $c$ is orthogonal to all code words of $\mc_{j,n-k+j}(n,q)$.
Since the scalar product is linear, is suffices that $c$ is orthogonal to the generators of $\mc_{j,n-k+j}(n,q)$.
By (1), this only requires that the scalar product of $c$ with a specific subspace of dimension at least $n-k+j$ is zero, e.g. the whole space.
This means that $c \cdot \one$ is zero.
Hence, $\mh_{j,k}(n,q) = \set{c \in \mc_{j,k}(n,q) : c \cdot \one = 0}$.

Since $c \cdot \one = 0$ is a linear equation, we know that $\set{c \in \mc_{j,k}(n,q) : c \cdot \one = 0}$ is a vector subspace of $\mc_{j,k}(n,q)$ of codimension 0 or 1.
Since we have proven in (1) that, for any $k$-space $\kappa$, $\kappa \cdot \one = 1$, this vector subspace must be a proper subspace, hence it has codimension 1, proving (3).

Now take two $k$-spaces $\kappa$ and $\kappa'$.
It is clear that $\kappa - \kappa' \in \mc_{j,k}(n,q)$.
If $\pi \in G_{n-k+j}$, then we know that $\kappa \cdot \pi = \kappa' \cdot \pi = 1$ by $(1)$.
Hence, $\pi \cdot (\kappa - \kappa') = 0$.
Therefore, $\kappa - \kappa'$ is orthogonal to all generators of $\mc_{j,n-k+j}(n,q)$, which means that $\kappa - \kappa' \in \mc_{j,n-k+j}(n,q)^\perp$.
As a result, if we fix $\kappa' \in G_k$, $K := \vspan{\kappa - \kappa' : \kappa \in G_k} \leq \mh_{j,k}(n,q)$.
Since $K \oplus \vspan {\kappa'} = \mc_{j,k}(n,q)$, the codimension of $K$ in $\mc_{j,k}(n,q)$ is at most one.
Thus, $\dim(K) \geq \dim\big(\mh_{j,k}(n,q)\big)$.
This is only possible if those spaces coincide.
\end{proof}

We can also say something about the code $\ms_{j,k}(n,q) := \vspan{\mc_{j,k}(n,q), \mc_{j,n-k+j}(n,q)^\perp}$.

\begin{lm}
\label{LmSpan}
\begin{enumerate}
 \item [(1)] $\dim \big( \ms_{j,k}(n,q) \big) = \dim \big( \mc_{j,n-k+j}(n,q)^\perp \big) + 1$.
 \item [(2)] $\ms_{j,k}(n,q) 
 = \mh_{j,n-k+j}(n,q)^\perp
 = \{v \in V(j,n,q) : (\exists \alpha \in \fp)(\forall \kappa \in G_{n-k+j})(v \cdot \kappa = \alpha) \}$.
 \item [(3)] The minimum weight code words of $\ms_{0,k}(n,q)$ are scalar multiples of $k$-spaces.
 \item [(4)] If $j \geq 1$, then the minimum weight code words of $\ms_{j,k}(n,q)$ lie in $\mc_{j,n-k+j}(n,q)^\perp$.
\end{enumerate}
\end{lm}

\begin{proof}
(1) By Grassmann's identity and Lemma \ref{LmHull} (3), we have
\begin{align*}
 \dim \big( \ms_{j,k}(n,q) \big) 
 &= \dim \big( \mc_{j,k}(n,q) \big) + \dim \big( \mc_{j,n-k+j}(n,q)^\perp \big) - \dim \big( \mc_{j,k}(n,q) \cap \mc_{j,n-k+j}(n,q)^\perp \big)\\
 &= \dim \big( \mc_{j,n-k+j}(n,q)^\perp \big) + 1.
\end{align*}

(2) Since $\vspan{A,B}^\perp = A^\perp \cap B^\perp$, we have that $\ms_{j,k}(n,q)^\perp = \mc_{j,k}(n,q)^\perp \cap \mc_{j,n-k+j}(n,q) = \mh_{j,n-k+j} (n,q) $.
By Lemma \ref{LmHull} (2), this means that $\ms_{j,k}(n,q)^\perp = \vspan{\kappa-\kappa' : \kappa, \kappa' \in G_{n-k+j}}^\perp$.
Hence, $v \in \ms_{j,k}(n,q) \Leftrightarrow (\forall \kappa, \kappa' \in G_{n-k+j})( v \cdot (\kappa-\kappa') = 0)$.
This means that $v \in \ms_{j,k}(n,q)$ if and only if $v\cdot \kappa$ is equal for all $(n-k+j)$-spaces $\kappa$.\\

(3) The arguments used in the literature to prove this exact same statement about $\mc_k(n,q)$ are also valid for the bigger code $\ms_{0,k}(n,q)$; for instance, see \cite[Proposition 1]{bagchi}.
The authors of the latter article make the exact same observation at the very end of their work.\\

(4) Assume that $j \geq 1$ and take a code word $c \in \ms_{j,k}(n,q)$, with $c \not \in \mc_{j,n-k+j}(n,q)^\perp$.
Then we know that there exists some $\alpha \in \fp^*$, with $c \cdot \kappa = \alpha$, for all $\kappa \in G_{n-k+j}$.
In particular, this means that every $(n-k+j)$-space $\kappa$ contains an element of $\supp(c)$.
Consider the set $V = \set{(\lambda,\kappa) : \lambda \in \supp(c), \lambda \subset \kappa \in G_{n-k+j}}$.
Since for every $\kappa$, there exists a $\lambda$ with $(\lambda,\kappa) \in V$, we get
\[
 \wt(c) \gauss{n-j}{k-j}
 = \wt(c) \gauss {n-j}{(n-k+j)-j}
 = |V|
 \geq \gauss {n+1}{(n-k+j)+1} = \gauss{n+1}{k-j}.
\]
Here we used the fact that $\gauss n k = \gauss n {n-k}$.
Manipulating this inequality yields
\begin{align*}
 \wt(c)
 &\geq \frac{\gauss{n+1}{k-j}}{\gauss{n-j}{k-j}}
 = \frac{\frac{(q^{n+1}-1)(q^n-1)\cdots(q^{n+2-k+j}-1)}{(q^{k-j}-1)(q^{k-j-1}-1)\cdots(q-1)}}{\frac{(q^{n-j}-1)(q^{n-j-1}-1)\cdots(q^{n-k+1}-1)}{(q^{k-j}-1)(q^{k-j-1}-1)\cdots(q-1)}}
 = \frac{q^{n+1}-1}{q^{n-j}-1} \frac{q^n-1}{q^{n-j-1}-1} \dots \frac{q^{n+2-k+j}-1}{q^{n-k+1}-1}\\
 & > (q^{j+1})^{k-j} \geq 2q^{k-j}.
\end{align*}
However, by Result \ref{ResMaxMinWtDualCode}, the minimum weight of $\mc_{j,n-k+j}(n,q)^\perp$ is at most $2q^{k-j}$.
Hence, the minimum weight code words of $\ms_{j,k}(n,q)$ must be contained in $\mc_{j,n-k+j}(n,q)^\perp$.
\end{proof}

Also note that, given a space $\pi$ with $\dim(\pi) > k$, $\pi^{(j)} = \sum_{\kappa \in G_k(\pi)} \kappa^{(j)}$.
This way, we see that if $k > k'$, then $\mc_{j,k}(n,q) \leq \mc_{j,k'}(n,q)$ and $\mc_{j,k}(n,q)^\perp \geq \mc_{j,k'}(n,q)^\perp$.


\section{Codes of points and $\boldsymbol{k}$-spaces}\label{SectPointsKSpaces}

The tool to guide us towards a characterisation of small weight code words of $\mc_k(n,q)$, is the following linear map.
It is essentially due to Lavrauw, Storme \& Van de Voorde \cite[Lemma 11]{lavrauwB}, but they only use it for a result regarding $\mc_k(n,q)^\perp$ (see Result \ref{ResEmbeddedCodeWords}).
We define it in a more general form, for all values of $j$.

\begin{df}
Take a point $R$ in $\pg(n,q)$ and a hyperplane $\pi$ not through $R$.

Define
\[
 \text{proj}^{(j)}_{R,\pi} : V(j,n,q) \rightarrow V(j,\pi): v \mapsto \bigg(\prj j R \pi v: \lambda \mapsto \sum_{ \lambda' \in G_j(\vspan{R,\lambda})} v(\lambda')\bigg)\textnormal{.}
\]
This means that the value in $\prj j R \pi v$ of a $j$-space $\lambda \subset \pi$ is the sum of the values in $c$ of all $j$-spaces $\lambda'$ in the $(j+1)$-space $\vspan{R,\lambda}$.
We could also write this as
\[
 \text{proj}^{(j)}_{R,\pi}(v)(\lambda)  = v \cdot \vspan{R,\lambda}^{(j)}
\]
If $j = 0$, we will denote $\prj 0 R \pi v$ by $\pr R \pi v$.
\end{df}

We will now give the most important properties of this map.

\begin{lm}
\label{LmProj}
Assume that $R$ is a point of $\pg(n,q)$, and $\pi$ is a hyperplane not through $R$. Then the following holds:
\begin{enumerate}
    \item[(1)] The map $\textnormal{proj}^{(j)}_{R,\pi}$ is linear.
    \item[(2)] If $k < n-1$, then $\prj j R \pi {\mc_{j,k}(n,q)} = \mc_{j,k}(n-1,q)$.
    \item[(3)] If $k > j+1$, then $\prj j R \pi {\mc_{j,k}(n,q)^\perp} = \mc_{j,k-1}(n-1,q)^\perp$.
    \item[(4)] If $v \in V(j,n,q)$ and $R \not \in \supp_0 (v)$, then $\wt(\prj j R \pi v) \leq \wt(v)$, with equality if and only if no $(j+1)$-space through $R$ contains more than one $j$-space of $\supp(v)$.
    \item[(5)] If $v \in V(j,n,q)$, then $v \cdot \one = \prj j R \pi v \cdot \one$.
\end{enumerate}
\end{lm}

\begin{proof}
(1) 
To prove that $\text{proj}^{(j)}_{R,\pi}$ is linear, we take $\alpha, \beta \in \fp$, and $v,w \in V(j,n,q)$.
We need to prove that $\prj j R \pi {\alpha v + \beta w} = \alpha \prj j R \pi v + \beta \prj j R \pi w$.
Take a $j$-space $\lambda \subset \pi$.
Then
\begin{align*}
 \prj j R \pi {\alpha v + \beta w}(\lambda) 
 & = (\alpha v + \beta w) \cdot \vspan{R,\lambda}
 = \alpha v \cdot \vspan{R,\lambda} + \beta w \cdot \vspan{R,\lambda} \\
 & = \alpha \prj j R \pi v (\lambda) + \beta \prj j R \pi w (\lambda).
\end{align*}
Since this holds for every $j$-space $\lambda \subset \pi$, this means that
$\prj j R \pi {\alpha v + \beta w} = \alpha \prj j R \pi v + \beta \prj j R \pi w$.\\

(2) 
Let $\kappa$ be a $k$-space of $\pg(n,q)$.
First, assume that $R \not \in \kappa$.
It is easy to see that $\prj j R \pi \kappa$ is the $k$-space $\vspan{R,\kappa} \cap \pi$.
So assume that $R \in \kappa$.
Take a $j$-space $\lambda \subset \pi$.
Then $\prj j R \pi \kappa (\lambda)$ equals the number of $j$-spaces in $\vspan{R,\lambda} \cap \kappa$.
Note that $\dim \big(\vspan{R,\lambda} \cap \kappa \big) = \dim(\lambda \cap \kappa) + 1$.
This implies that
\[
 \prj j R \pi \kappa (\lambda) = 
 \begin{cases}
  1 & \text{if } \dim(\lambda \cap \kappa) \geq j-1,\\
  0 & \text{otherwise.}
 \end{cases}
\]
The number of $k$-spaces $\kappa'$ in $\pi$ through a $j$-space $\lambda$, containing the $(k-1)$-space $\kappa \cap \pi$ equals 0 if $\dim(\lambda \cap \kappa) < j-1$, equals 1 if $\dim(\lambda \cap \kappa) = j-1$, and equals $\gauss{(n-1)-(k-1)}{k-(k-1)} \equiv 1 \pmod p$ if $\dim(\lambda \cap \kappa) = j$.
Thus, 
\[
 \prj j R \pi \kappa = \sum_{\substack{\kappa' \in G_k(\pi) \\ \kappa \cap \pi \subset \kappa'}} \kappa' \in \mc_{j,k}(n-1,q).
\]
Therefore the map $\text{proj}^{(j)}_{R,\pi}$ maps the set $G_k(n,q)^{(j)}$, which generates the code $\mc_{j,k}(n,q)$, to a subset of $\mc_{j,k}(n-1,q)$, containing its generating set $G_k(\pi)^{(j)}$.
Since this map is linear, this proves that $\prj j R \pi {\mc_{j,k}(n,q)} = \mc_{j,k}(n-1,q)$.\\

(3) Take $c \in \mc_{j,k}(n,q)^\perp$.
To prove that $\prj j R \pi c \in \mc_{j,k-1}(n-1,q)^\perp$, we need to prove that $\prj j R \pi c \cdot \kappa = 0$ for every $(k-1)$-space $\kappa \subset \pi$.
\begin{align*}
 \prj j R \pi c \cdot \kappa
 & = \sum_{\lambda \in G_j(\pi)} \prj j R \pi c (\lambda) \cdot \kappa(\lambda)
 = \sum_{\substack{\lambda \in G_j(\pi) \\ \lambda \subset \kappa }} \sum_{ \lambda' \in G_j(\vspan{R,\lambda} )} c(\lambda')\\
 & = \sum_{ \lambda' \in G_j(\vspan{R,\kappa} )} c(\lambda') \sum_{ \substack{\lambda \in G_j(\kappa) \\ \lambda' \subset \vspan{R,\lambda}} } 1.
\end{align*}
For a fixed $j$-space $\lambda'$ in $\vspan{R,\kappa}$, we have
\[
 \sum_{ \substack{\lambda \in G_j(\kappa) \\ \lambda' \subset \vspan{R,\lambda}} } 1 =
 \begin{cases}
 1 & \text{if } R \not \in \lambda', \\
 \theta_{k-j-1} & \text{otherwise}
 \end{cases}
 \quad \equiv 1 \pmod p.
\]
Therefore,
\[
 \prj j R \pi c \cdot \kappa
 = \sum_{\lambda' \in G_j(\vspan{R,\kappa} )} c(\lambda')
 = c \cdot \vspan{R,\kappa}
 = 0,
\]
because $\vspan{R,\kappa}$ is a $k$-space and $c \in \mc_{j,k}(n,q)^\perp$.
Hence, $\prj j R \pi {\mc_{j,k}(n,q)^\perp} \leq \mc_{j,k-1}(n-1,q)^\perp$.
To prove that equality holds, we can embed a code word $c'$ of $\mc_{j,k-1}(n-1,q)^\perp$ in $\pi$ (see Construction \ref{ConstrEmbedding}).
The image of this embedded code word under $\text{proj}^{(j)}_{R,\pi}$ will again be $c'$. \\

(4) It holds that if $\lambda \in \supp(\prj j R \pi v)$, then the $(j+1)$-space $\vspan{R, \lambda}$ must contain a $j$-space of $\supp(v)$.
Hence, if $R \not \in \supp_0(v)$, every $j$-space in $\supp(c)$ lies in a unique $(j+1)$-space through $R$, which implies that the number of $(j+1)$-spaces through $R$ that contains an element of $\supp(v)$ is at most $\wt(v)$.
Thus, $\wt(\prj j R \pi v) \leq \wt(v)$.
It is easy to see that equality holds if and only if no $(j+1)$-space through $R$ contains more than one element of $\supp(v)$.\\

(5)
\begin{align*}
    \prj j R \pi v \cdot \one 
    &= \sum_{\lambda \in G_j(\pi)} \pr R \pi v (\lambda) \cdot 1 = \sum_{\lambda \in G_j(\pi)} \sum_{ \lambda' \in G_j(\vspan{R,\lambda} )} v(\lambda')
    = \sum_{\lambda' \in G_j(n,q)} v(\lambda') \sum_{\substack{\lambda \in G_j(\pi) \\ \lambda' \subset \vspan{R,\lambda} }} 1 \\
    & = \sum_{\substack{\lambda' \in G_j(n,q) \\ R \not \in \lambda'}} v(\lambda') + \gauss{(n-1)-(j-1)}{j-(j-1)} \sum_{\substack{\lambda' \in G_j(n,q) \\ R \in \lambda'}} v(\lambda') \\
    & \equiv \sum_{\substack{\lambda' \in G_j(n,q) \\ R \not \in \lambda'}} v(\lambda') + \sum_{\substack{\lambda' \in G_j(n,q) \\ R \in \lambda'}} v(\lambda')
    = v \cdot \one \pmod p.
     \qedhere
\end{align*}
\end{proof}

\begin{rmk}\label{RmkProjHyperplane}
When constructing $\pr R \pi c$, what we are actually doing is projecting from the point $R$ onto a hyperplane $\pi$.
One could also view this as working in the quotient geometry of $\pg(n,q)$ through $R$.
This way we see that the choice of $\pi$ is not really relevant.
In other words, for any two choices of hyperplanes $\pi_1, \pi_2 \not \ni R$ in $\pg(n,q)$, the nature of the code words $\pr R {\pi_1} c$ and $\pr R {\pi_2} c$ will essentially stay the same.
More rigorously, there exists a collineation $\beta$ from $\pi_1$ to $\pi_2$ such that $\pr R {\pi_1} c (\lambda) = \pr R {\pi_2} c (\lambda^\beta)$, for every $\lambda \in G_j(\pi_1)$.
This collineation $\beta$ maps a subspace $\lambda$ of $\pi_1$ to $\vspan{R,\lambda} \cap \pi_2$.
The reason that we emphasize which hyperplane is considered is solely to obtain a natural embedding of $\supp(\pr R \pi c)$ in $\pg(n,q)$.

As such, when considering $\pr R \pi c$, we can, at any time and w.l.o.g., choose $\pi$ to be any other hyperplane not containing $R$.
\end{rmk}

Eventually, we will use this map to characterise low weight code words of $\mc_k(n,q)$.
However, we first need a few important lemmas, some of which are tedious to prove.

\begin{lm}
 \label{LmSmallestWtThreeKspaces}
Let $c \in \mc_k(n,q)$ be a linear combination of three $k$-spaces, which can't be written as a linear combination of at most two $k$-spaces.
Then $\wt(c) > W(k,q)$.
\end{lm}

\begin{proof}
Let $\kappa_i$ ($i=1,2,3$) be three distinct $k$-spaces of which $c$ is a linear combination.
We write $\sigma = \bigcap_{i=1}^3 \kappa_i$, $K = \vspan{\kappa_1, \kappa_2, \kappa_3}$, and $s = \dim(\sigma)$.
A simple but tedious argument to prove this result is finding a lower bound on $\wt(c)$ that exceeds $W(k,q)$.
This is done by counting points that lie in precisely one of the three $k$-spaces $\kappa_i$, as such points are necessarily contained in $\supp(c)$.
As the proof involves a case-by-case analysis of the geometric nature of these $k$-spaces, we will omit most details of the easier cases.

If $s = k-1$, one can prove rather easily that $\wt(c) \in \set{3q^k, 3q^k + \theta_{k-1}}$.

If $s = k-2$, there are two cases to consider.
In the first case, we assume that two $k$-spaces intersect in $\sigma$.
Hence, each of these two $k$-spaces contains at least $\theta_k-\theta_{k-1}$ points not lying in any other of the three spaces.
As the third space adds at least $\theta_k-\theta_{k-1}-(\theta_{k-1}-\theta_{k-2})$ points of $\supp(c)$ we haven't considered before, we obtain $\wt(c) \geq 3q^k-q^{k-1}$.
In the second case, we assume that each two $k$-spaces intersect in a $(k-1)$-space.
As such, the set $\{\kappa_1,\kappa_2,\kappa_3\}$ forms an Erd\H{o}s-Ko-Rado set, implying that $K$ is $(k+1)$-dimensional.
Hence, we can consider the restriction of the code word $c$ to $K$ and rely on Result \ref{ResLins}.

Finally, assume that $s \leq k-3$.
Denote $\sigma_2 = \kappa_1 \cap \kappa_2$ and $\sigma_3 = \kappa_1 \cap \kappa_3$.
We know that $\dim(\sigma_2 \cap \sigma_3) = \dim(\sigma) = s$, and that $\dim\left(\vspan{\sigma_2, \sigma_3}\right) \leq \dim(\kappa_1) = k$.
Grassmann's identity implies that $\dim(\sigma_2) + \dim(\sigma_3) \leq k + s$.
We also know that the dimension of $\sigma_2$ and $\sigma_3$ are at most $k-1$.
Note that if $a \geq b$, then $\theta_a + \theta_b < \theta_{a+1} + \theta_{b-1}$.
Keeping this in mind, together with $\dim(\sigma_2) + \dim(\sigma_3) \leq k + s$, we know that $\sigma_2 \cup \sigma_3$ contains at most $\theta_{k-1} + \theta_{s+1} - \theta_{s} = \theta_{k-1} + q^{s+1} \leq \theta_{k-1} + q ^{k-2}$ points.
Hence, $\kappa_1$ contains at least $\theta_k - \theta_{k-1} - q^{k-2} = q^k - q^{k-2}$ points outside of $\kappa_2 \cup \kappa_3$.
Repeating this argument for each of the two other $k$-spaces, we obtain $\wt(c) \geq 3(q^k - q^{k-2})$.
\end{proof}

\begin{df}
Let $S$ be a point set in $\pg(n,q)$.
If a line $l \subseteq \pg(n,q)$ intersects $S$ in at most $2$ points, we will call $l$ a \emph{short secant} to $S$.
If $l$ intersects $S$ in at least $q$ points, we will call $l$ a \emph{long secant} to $S$.
\end{df}

\begin{lm}\label{LmShortAndLongSecants}
Let $c$ be a code word of $\mc_k(n,q)$ with $q \geq 5$ and $\wt(c) \leq W(k,q)$.
\begin{enumerate}
    \item[(1)] All lines in $\pg(n,q)$ are either short or long secants to $\supp(c)$.
    \item[(2)] $c \cdot s = \begin{cases}
     c \cdot \one \quad &\textnormal{if } s \textnormal{ is a } 2 \textnormal{-secant to } \supp(c) \textnormal{,}\\
     0 \quad &\textnormal{if } s \textnormal{ is a } q \textnormal{-secant to } \supp(c) \textnormal{.}
    \end{cases}$
\end{enumerate}
\end{lm}

\begin{proof}
We will prove this by induction on $n$.
If $n = k + 1$, then we know, by Result \ref{ResLins}, that $c$ is a linear combination of at most two $k$-spaces.
In particular, this implies that $\supp(c)$ is either equal to the empty set, a $k$ space, or the union or symmetric difference of two $k$-spaces, proving the first statement of the lemma.
If $s$ is a $2$-secant to $\supp(c)$, then $c$ must be a linear combination of precisely two $k$-spaces.
Then both $c \cdot s$ and $c \cdot \one$ equal the sum of the coefficients arising from this linear combination.
If $s$ is a $q$-secant to $\supp(c)$, then $c$ must be a scalar multiple of the difference of two distinct $k$-spaces.
A $q$-secant can only exist in this setting if
$c$ takes the same non-zero value in all but one point of $s$.
Hence, $c \cdot s = 0$, proving the second statement.

As such, let us assume that $n \geq k + 2$ and that the lemma is true for all code words in $\mc_k(n-1,q)$ with weight at most $W(k,q)$.
Note that, by Lemma \ref{LmProj} (4), the induction hypothesis implies that both statements of this lemma hold for the code word $\pr R \pi c$, for any point $R \notin \supp(c)$ and any hyperplane $\pi \not\ni R$.

Suppose that $s$ is an $m$-secant to $\supp(c)$ and suppose that every plane through $s$ intersects $\supp(c)$ in at least $m+3$ points.
Then $\wt(c) \geq 3\theta_{n-2} + m \geq 3\theta_k > W(k,q)$, a contradiction.
Hence, there exists a plane $\sigma$ such that $|\sigma \cap \supp(c)| \leq m+2$.
Let $\pi$ be a hyperplane intersecting $\sigma$ in $s$.\\

(1) Let $3 \leq m \leq q-1$.
    To find a contradiction and prove the first part of the lemma, we distinguish three cases depending on the value of $|\sigma\cap\supp(c)|\in\{m,m+1,m+2\}$.
    For each of these cases, one can find a point $R \in \sigma \setminus s$ such that $s$ contains precisely $m$ or $m+1$ points (if $m \neq q-1$), or $m$ or $m-1$ points (if $m \neq 3$) of $\supp(\pr R \pi c)$.
    Hence, each of these cases results in the existence of a secant to $\supp(\pr R \pi c)$ that is neither short nor long, contradicting the induction hypothesis.
    We leave the rather tedious details of this case-by-case proof to the reader.\\
    
(2) Let $m \in \{2,q\}$.
    The proof of the second statement can easily be obtained if we know that $\sigma\cap\supp(c) \subseteq s$.
    Indeed, if the latter would be the case, then $s$ would be an $m$-secant to $\supp(\pr R \pi c)$ for any choice of $R \in \sigma \setminus s$.
    Moreover, as all lines through $R$ in $\sigma$ contain at most one point of $\supp(c)$, we know that $c \cdot s = \pr R \pi c \cdot s$.
    By the induction hypothesis and Lemma \ref{LmProj} (5), we know that
    \[
        \pr R \pi c \cdot s = \begin{cases}
         \pr R \pi c \cdot \one = c \cdot \one \quad &\textnormal{if } s \textnormal{ is a } 2 \textnormal{-secant to } \supp(c) \textnormal{,}\\
         0 \quad &\textnormal{if } s \textnormal{ is a } q \textnormal{-secant to } \supp(c) \textnormal{.}
        \end{cases}
    \]
    
    So let us assume, on the contrary, that $|\sigma \cap \supp(c)| \in \{m+1,m+2\}$.
    
    If $m = 2$, we can find a point $R \in \sigma \setminus (s \cup \supp(c))$ such that $s$ contains precisely $|\sigma \cap \supp(c)| < q$ points of $\supp(\pr R \pi c)$, contradicting the assumptions.
    
    Let $m = q$ and let $O$ be the unique point in $s \setminus \supp(c)$.
    Let $t$ be a line of $\sigma$ through $O$ containing a point of $(\sigma \cap \supp(c)) \setminus s$.
    Then all points of $(\sigma \cap \supp(c)) \setminus s$ have to lie on $t$, as else we can find a $3$-secant to $\supp(c)$ in $\sigma$, contradicting (1).
    In this way, if we choose $Q \in t \cap \supp(c)$, $QP$ is a $2$-secant to $\supp(c)$ for every choice of $P \in s \setminus \{O\}$.
    As we already proved this statement in case $m = 2$, we know that all values $c \cdot QP$ are the same, for every choice of $P \in s \setminus \{O\}$.
    As $c \cdot QP = c(Q) + c(P)$, this means that $c$ takes the same value in every point of $s \setminus \{O\}$, resulting in $c \cdot s = 0$.
\end{proof}

\begin{lm}\label{LmPointSet01qq+1}
 Assume that $S$ is a point set in $\pg(n,q)$, $q \geq 4$, with the property that every line intersects $S$ in $0$, $1$, $q$ or $q+1$ points.
 Then there exists a hyperplane $H$ in $\pg(n,q)$ such that either $S \subseteq H$ or $S^c \subseteq H$, where $S^c$ denotes the complement of $S$ in $\pg(n,q)$.
\end{lm}

\begin{proof}
We prove this by induction on $n$.
Note that it is trivial for $n=1$.
Now assume that it holds in $\pg(n-1,q)$, we will prove that it holds in $\pg(n,q)$.
The induction hypothesis implies that for every hyperplane $\pi$ of $\pg(n,q)$, either $S \cap \pi$ or $S^c \cap \pi$ is contained in an $(n-2)$-space of $\pi$.
If $S$ spans $\pg(n,q)$, then we can take a hyperplane $\pi$ spanned by $n$ points of $S$ and a point $P \in S \setminus \pi$.
By the induction hypothesis, $S^c \cap \pi$ is contained in an $(n-2)$-space in $\pi$.
Therefore, there are at least $q^{n-1}$ lines through $P$ intersecting $\pi$ in a point of $S$.
These lines contain at least $q$ points of $S$, yielding that $|S| \geq q^{n-1}(q-1) + 1$.
Note that this lemma is self-dual in the sense that if we replace $S$ by $S^c$, the lemma stays the same.
Thus, if $S^c$ spans $\pg(n,q)$, then $|S^c| \geq q^{n-1}(q-1) + 1$.
Hence, if both $S$ and $S^c$ span $\pg(n,q)$, then
\[
 \theta_n = |S| + |S^c| \geq 2(q^{n-1}(q-1) + 1),
\]
a contradiction if $q \geq 4$.
Therefore, either $S$ or $S^c$ is contained in a hyperplane.
\end{proof}

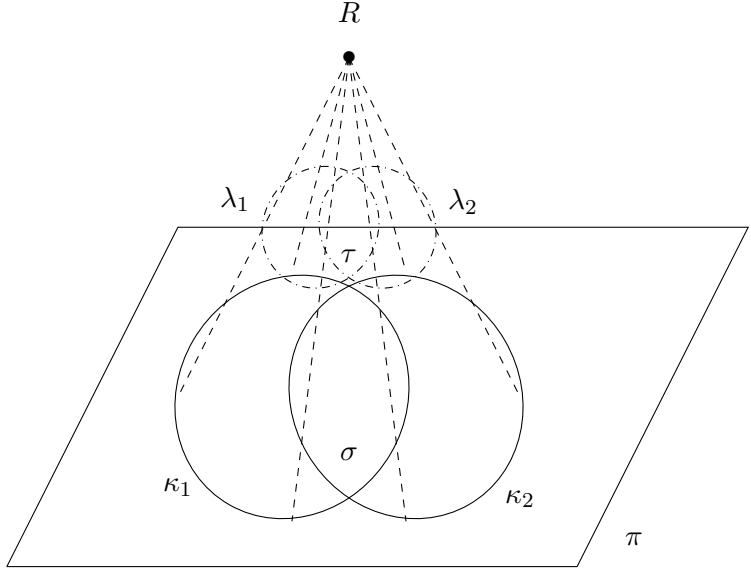
\begin{center}
 \begin{tikzpicture}[scale=1.5]
 \draw (.5,0) -- (5.5,0) -- (7,3) -- (2,3) -- (.5,0);
  \draw (6,.25) node{$\pi$};
 \draw (3,1.5) circle [x radius=1, y radius=1.1, rotate=-30];
  \draw (2,.7) node{$\kappa_1$};
 \draw (4,1.5) circle [x radius=1, y radius=1.1, rotate=30];
  \draw (5,.6) node{$\kappa_2$};
  \draw (3.5,1) node{$\sigma$};
 \draw (3.5,4.5) node{\textbullet};
  \draw (3.5,4.9) node{$R$};
 \draw[dashed] (3.5,4.5) -- (3,.4);
 \draw[dashed] (3.5,4.5) -- (4,.4);
 \draw[dashed] (3.5,4.5) -- (2,1.5);
 \draw[dashed] (3.5,4.5) -- (5,1.5);
 \draw[dashed] (3.5,4.5) -- (3,2.6);
 \draw[dashed] (3.5,4.5) -- (4,2.6);
 \draw[dashdotted] (3.75,3) circle [x radius=.5, y radius=.55, rotate=30];
  \draw (2.5,3.25) node{$\lambda_1$};
 \draw[dashdotted] (3.25,3) circle [x radius=.5, y radius=.55, rotate=-30];
  \draw (4.5,3.23) node{$\lambda_2$};
 \draw (3.5,2.75) node {$\tau$};
\end{tikzpicture}
\end{center}

\begin{lm}\label{LmAtLeastHalfPointsSameValue}
 Let $c$ be a code word of $\mc_k(n,q)$ with $q \geq 5$ and $\wt(c) \leq W(k,q)$, and assume that all code words of $\mc_k(n-1,q)$ with weight at most $W(k,q)$ are linear combinations of at most two $k$-spaces.
 Consider a point $R \notin \supp(c)$ and a hyperplane $\pi \not \ni R$; let $\kappa_1, \kappa_2 \in G_k(\pi)$, $\kappa_1\neq\kappa_2$, and let $\alpha_1, \alpha_2 \in \mathbb{F}_p^*$.
 Define $\lambda_i := \vspan{R, \kappa_i}$ and $\tau := \lambda_1\cap\lambda_2$.
 Assume that precisely one of the following holds:
 \begin{enumerate}
     \item[(1)] $q$ is even and $\pr R \pi c = \kappa_1$,
     \item[(2)] $\pr R \pi c = \alpha_1\kappa_1 + \alpha_2\kappa_2$.
 \end{enumerate}
 Then there exists a $k$-space $H$ such that more than $\frac{1}{2}\theta_{k}$ points of $H$ have the same non-zero value in $c$.
\end{lm}

\begin{proof}
Remark that, by Lemma \ref{LmProj} (2, 4), the assumptions imply that $\pr {R'} {\pi'} c$ is a linear combination of at most two $k$-subspaces of $\pi'$, for every point $R' \notin \supp(c)$ and every hyperplane $\pi' \not\ni R$.

First, assume that (2) holds.
We will make two observations, the first one is stated as follows.
\begin{obs}\label{ObsTangent}
 Every line in $\lambda_1 \setminus \tau$ through $R$ is tangent to $\supp(c)$.
\end{obs}
Indeed, take such a line $l$.
We know that $\alpha_1 = \pr R \pi c (l \cap \pi) = c \cdot l$.
By Lemma \ref{LmShortAndLongSecants}, $l$ is either a short or a long secant to $\supp(c)$.
By that same lemma, $l$ cannot be a $0$- or a $q$-secant, as else $\alpha_1 = 0$.
Finally, $l$ cannot be a $2$-secant either, as else, by Lemma \ref{LmShortAndLongSecants} and Lemma \ref{LmProj}, $\alpha_1 = c \cdot l = c \cdot \one = \pr R \pi c \cdot \one = \alpha_1 + \alpha_2$, which would imply that $\alpha_2=0$.

\begin{obs}\label{ObsTwoSecant}
    All $2$-secants to $\supp(c)$ in $\lambda_1$ are contained in $\tau$.
\end{obs}
Let $s$ be a $2$-secant to $\supp(c)$ in $\lambda_1$ that is not contained in $\tau$.
Take a point $S \in s \setminus \tau$.
By Remark \ref{RmkProjHyperplane}, we can choose $\pi$ to be a hyperplane not through $R$, intersecting $s$ in $S$.
Note that this also means that $s$ intersects $\kappa_1$ in $S$.
As $q > 2$, we can choose a point $R_1 \in s \setminus(\supp(c) \cup \tau)$.
By Observation \ref{ObsTangent}, as $R_1 \in \lambda_1 \setminus \tau$, $RR_1$ is tangent to $\supp(c)$ and hence the unique point of $\supp(c)$ on $RR_1$ must have value $\alpha_1$.
Denote $T = RR_1 \cap \kappa_1$.

In this way, we can see that
\begin{itemize}
    \item $\pr {R_1} \pi c (S) = \alpha_1 + \alpha_2$, by Lemma \ref{LmShortAndLongSecants} and Lemma \ref{LmProj} (5).
    \item $\pr {R_1} \pi c (T) = \alpha_1$, implying in particular that $\pr {R_1} \pi c \neq \zero$.
\end{itemize}

Therefore, $\pr {R_1} \pi c$ takes distinct non-zero values and must also be a linear combination of exactly two distinct $k$-spaces.

It's clear that $\pr {R_1} \pi c$ and $\pr R \pi c$ cannot share the same $k$-subspaces of $\pi$, as else the points $S, T \in \kappa_1 \setminus \tau$ must have the same value w.r.t.\ $\pr {R_1} \pi c$, resulting in $\alpha_1 = \alpha_1 + \alpha_2$, a contradiction.
Hence, we can find a $k$-space $\kappa_3 \notin \{\kappa_1, \kappa_2\}$ in $\pi$ containing, by Observation \ref{ObsTangent}, at least $q^k$ points in a $k$-dimensional affine subspace, each connected to $R_1$ by a tangent line to $\supp(c)$.

One can observe at least $q^k - 2q^{k-1} + \theta_{k-2}$ points of $\supp(c)$ outside of $\lambda_1 \cup \lambda_2$.
Hence, we get the following contradiction: $\wt(c) \geq |(\lambda_1 \cup \lambda_2 )\cap \supp(c)| + |\lambda_3 \setminus (\lambda_1 \cup \lambda_2) \cap \supp(c)| \geq 2q^k + q^k - 2q^{k-1} + \theta_{k-2} = 3q^k - 2q^{k-1} + \theta_{k-2} > W(k,q)$. \\

Define $\mathcal{S} := (\lambda_1 \setminus \tau) \cap \supp(c)$.
By Lemma \ref{LmShortAndLongSecants}, Observation \ref{ObsTwoSecant} and Lemma \ref{LmPointSet01qq+1}, there exists a $k$-space $H$ in $\lambda_1$ such that either $\mathcal{S} \subseteq H$ or $\big(\lambda_1 \setminus \mathcal{S}\big) \subseteq H$.
The latter would imply that $\wt(c) \geq |\lambda_1 \setminus (H \cup \tau)| \geq q^{k+1} - q^k > W(k,q)$ as $q \geq 5$, a contradiction.
Thus, $\mathcal{S} \subseteq H$ must be valid.
By Observation \ref{ObsTangent}, all $q^k > \frac{1}{2}\theta_k$ points in $\mathcal{S}$ have non-zero value $\alpha_1$ in $c$, proving the lemma. \\

Now assume that (1) holds.
The proof stays mainly the same, except for the proof of Observation \ref{ObsTwoSecant2}; we will indicate what arguments need to be changed or added in order to keep all proofs valid.
In general, every instance of $\alpha_1$ and $\alpha_2$ can be replaced by $1$, as $q$ is even, and every instance of $\kappa_2$ and $\tau$ need to be replaced by $\emptyset$.
As such, Observation \ref{ObsTangent} becomes the following statement:
\begin{obs}\label{ObsTangent2}
    Every line in $\lambda_1$ through $R$ is tangent to $\supp(c)$.
\end{obs}
This can be proven using exactly the same arguments as before: such a line $l$ can only be a tangent line or a $2$-secant, and if $l$ is a $2$-secant, we would obtain $1 = \alpha_1 = c \cdot l = 1 + 1 = 0$, as $q$ is even, a contradiction.

Observation \ref{ObsTwoSecant} changes to the following:
\begin{obs}\label{ObsTwoSecant2}
    There are no $2$-secants to $\supp(c)$ contained in $\lambda_1$.
\end{obs}
We can repeat all notations and arguments used to prove Observation \ref{ObsTwoSecant} (keeping in mind that $\tau$ is replaced by $\emptyset$) and prove that there exists a $k$-space $\kappa_3\neq\kappa_1$ in $\pi$ in which, by Observation \ref{ObsTangent2}, each point is connected to $R_1$ by a tangent line to $\supp(c)$.

Remark that, as $q$ is even, $\pr {R_1} \pi c (S) = 0$, implying that $S \notin \kappa_3$ as $\pr {R_1} \pi c (Q) = 1$ for every $Q \in \kappa_3$.
As such, for each point $P$ of the at least $\theta_k - \theta_{k-1} = q^k$ points of $\supp(c)$ in $\lambda_3 := \vspan{R_1, \kappa_3}$ not contained in $\lambda_1$, the plane $\sigma_P := \vspan{s, P}$ intersects $\lambda_1$ in the $2$-secant $s$ and $\lambda_3$ in the tangent line $R_1P$ (Observation \ref{ObsTangent2}).
If $|\sigma_P \cap \supp(c)| \leq 4$, then a clever choice of a point $R_2 \in \sigma_P \setminus \supp(c)$ (and a hyperplane $\pi_2 \not \ni R_2$) will result in the existence of a $|\sigma_P \cap \supp(c)|$-secant to $\supp(\pr {R_2} {\pi_2} c)$, contradicting Lemma \ref{LmShortAndLongSecants} as $q \geq 5$.
    
In conclusion, for every such point $P$, we find at least $2$ points of $\supp(c)$ outside of $\lambda_1 \cup \lambda_3$ by considering the plane $\sigma_P$.
As $R_1P$ is tangent to $\supp(c)$, each choice of such a $P$ will result $2$ extra points we haven't considered before.
Hence, $\wt(c) \geq |\lambda_1 \cap \supp(c)| + 3|(\lambda_3 \setminus \lambda_1) \cap \supp(c)| \geq \theta_k + 3q^k = 4q^k + 3\theta_{k-1} > W(k,q)$, a contradiction. \\

Given Observation \ref{ObsTangent2} and \ref{ObsTwoSecant2}, we can repeat the same arguments as before to conclude the proof.
\end{proof}

\setcounter{obs}{0}

\begin{thm}
\label{ThmPointsKSpaces}
If $c$ is a code word of $\mc_k(n,q)$, with $\wt(c) \leq W(k,q)$, then $c$ is a linear combination of at most two $k$-spaces.
Moreover, if $q \in Q_3 \cup Q_4 \cup Q_5$, then this bound is tight.
\end{thm}

\begin{proof}
The proof will be done by induction on $n$.
The case $n = k+1$ is Result \ref{ResLins}.
So assume that $n \geq k+2$ and that the theorem holds for the code $\mc_k(n-1,q)$.
Assume to the contrary that there exist code words of $\mc_k(n,q)$, with weight at most $W(k,q)$, which can't be written as a linear combination of at most two $k$-spaces.
Let $c$ be such a code word of smallest possible weight.
We will derive a contradiction by making use of the following observation.

\begin{obs}
\label{ObsContradiction}
There cannot exist a $k$-space $\kappa$ such that more than $\frac{1}{2}\theta_k$ points of $\kappa$ have the same non-zero value $\alpha$ in $c$.
\end{obs}

This follows from the fact that if such a $k$-space $\kappa$ would exist, then $\wt(c-\alpha \kappa) < \wt(c)$.
Since $c - \alpha \kappa \in \mc_k(n,q)$, this would mean that $c - \alpha \kappa$ is a linear combination of at most two $k$-spaces.
This is only possible if $c$ is a linear combination of three $k$-spaces.
But then $\wt(c) > W(k,q)$, by Lemma \ref{LmSmallestWtThreeKspaces}, a contradiction.\\

Given a hyperplane $\pi$ and a point $R \not \in \pi \cup \supp(c)$, there are three possibilities for $\pr R \pi c$:
\begin{itemize}
    \item [(P0)] $\pr R \pi c = \zero$.
    \item [(P1)] $\pr R \pi c = \alpha \kappa$, with $\alpha \in \fp^*$ and $\kappa$ a $k$-space of $\pi$.
    \item [(P2)] $\pr R \pi c = \alpha_1 \kappa_1 + \alpha_2 \kappa_2$, with $\alpha_i \in \fp^*$, and $\kappa_i$ distinct $k$-spaces of $\pi$.
\end{itemize}
This follows from the fact that $\wt( \pr R \pi c) \leq \wt(c) \leq W(k,q)$ (Lemma \ref{LmProj} (4)), hence due to the induction hypothesis, $\pr R \pi c$ is characterised as a linear combination of at most two $k$-spaces.\\

\underline{Case 1: Possibility (P2) never occurs.}

Take a point $P \in \supp(c)$, then there exists a tangent line $l$ to $\supp(c)$ through $P$.
Otherwise, each of the $\theta_{n-1}$ lines through $P$ contains another point of $\supp(c)$, implying that $\wt(c) > \theta_{n-1} > W(k,q)$, since $n \geq k + 2$, a contradiction.
Now take a point $R \in l \setminus \set P$ and a hyperplane $\pi$ with $\pi \cap l = \set P$.
Then $\pr R \pi c (P) = \sum_{Q \in PR} c(Q) = c(P)$.
Hence, $\pr R \pi c$ can't be $\zero$, which means (P1) is the only possibility.
So $\pr R \pi c = \alpha \kappa$ for some $\alpha \in \fp^*$, and some $k$-space $\kappa$.
It now follows that $\alpha = c(P)$ and $\pr R \pi c \cdot \one = \alpha$, so by Lemma \ref{LmProj} (5), $c(P) = c \cdot \one$.
Since this holds for all points of $\supp(c)$, they all have the same non-zero value $\alpha := c \cdot \one$ in $c$.
Note that this also means that $\pr R \pi c \cdot \one$ can never be zero, which means that possibility (P0) doesn't occur, for any choice of a hyperplane $\pi$ and a point $R \not \in \pi \cup \supp(c)$.

Taking an arbitrary hyperplane $\pi$ and a point $R \not \in \pi \cup \supp(c)$, we conclude that $\pr R \pi c = \alpha \kappa$, for some $k$-space $\kappa$ in $\pi$.
Define $\lambda := \vspan{R,\kappa}$.
For every point $P \in \kappa$, the line $PR$ intersects $\supp(c)$.
Therefore, the $(k+1)$-space $\lambda$ intersects $\supp(c)$ in at least $\theta_k$ points.

\bigskip
Remark that, if $q \geq 5$ and $q$ is even, Lemma \ref{LmAtLeastHalfPointsSameValue} can be used to obtain a contradiction to Observation \ref{ObsContradiction}.
As such, we can assume that $q$ is $2$, $4$ or odd.

Since $k \leq n - 2$, there exists a hyperplane $\pi'$ through $\lambda$.
Take a point $R' \not \in \pi' \cup \supp(c)$, then $\pr {R'} {\pi'} c = \alpha \kappa'$ for some $k$-space $\kappa'$ in $\pi'$.
We define the following numbers:
\begin{align*}
    x_1 = | \supp(c) \cap \pi' | \geq \theta_k, & &
    x_2 = | (\supp(c) \cap \pi') \setminus \kappa' |, & &
    x_3 = | \kappa' \setminus \supp(c) |.
\end{align*}

If $P \in (\supp(c) \cap \pi') \setminus \kappa'$, then
\[
 0 = \pr {R'} {\pi'} c (P) = \sum_{Q \in PR'} c(Q) \equiv \alpha \cdot |\supp(c) \cap PR'| \pmod p.
\]
Hence, $PR'$ contains $0 \pmod p$ points of $\supp(c)$, which means $PR'$ contains at least $p-1$ points of $\supp(c) \setminus \pi'$.
Remark that, if $q$ is odd and $q \neq 3$, then $p > 2$ and we can apply Lemma \ref{LmShortAndLongSecants} to state that $PR'$ contains at least $q-1$ points of $\supp(c) \setminus \pi'$.
If $P \in \kappa' \setminus \supp(c)$, then $PR'$ contains at least one point of $\supp(c) \setminus \pi'$. This yields
\begin{equation}
\label{EqCase1}
\begin{cases}
 (p-1) x_2 + x_3 \leq | \supp(c) \setminus \pi' | = \wt(c) - x_1 \leq 2 \theta_k - \theta_k = \theta_k & \text{ if } q \leq 4,\\
 (q-1) x_2 + x_3 \leq | \supp(c) \setminus \pi' | = \wt(c) - x_1 \leq W(k,q) - \theta_k & \text{ if } q > 4 \text{ is odd.}\\
\end{cases}
\end{equation}

Also note that $| \kappa' \cap \supp(c) | = x_1 - x_2$ and $x_3 = | \kappa' | - | \kappa' \cap \supp(c) | = \theta_k - x_1 + x_2$.
Hence the system of equations \eqref{EqCase1} becomes
\[
\begin{cases}
 (p-1) x_2 + \theta_k - x_1 + x_2 \leq \theta_k & \text{ if } q \leq 4,\\
 (q-1) x_2 + \theta_k - x_1 + x_2 \leq 3q^k-2q^{k-1}+\theta_{k-2}-1-\theta_k & \text{ if } q > 4 \text{ is odd,}\\
\end{cases}
\]
which implies
\[
x_2 \leq
\begin{cases}
 \frac{x_1}{p} & \text{ if } q \leq 4,\\
 \frac{x_1}{q} + q^{k-1} & \text{ if } q > 4 \text{ is odd,}\\
\end{cases}
\]
Thus, if $q \leq 4$, we get
\begin{equation}
\label{EqCase12}
 | \supp(c) \cap \kappa' | = x_1 - x_2 \geq \frac{p-1}{p} x_1 \geq \frac{p-1}{p} \theta_k.
\end{equation}
If $p = 2$, then $\theta_k$ is odd, hence $| \supp(c) \cap \kappa' | > \frac{1}{2} \theta_k$ since the left-hand side must be an integer.
Otherwise, $q = p = 3$ and $\frac{p-1}{p} = \frac{2}{3}$, which also implies $| \supp(c) \cap \kappa' | > \frac{1}{2} \theta_k$.
This yields a contradiction by Observation \ref{ObsContradiction}, since all points of $\supp(c)$ have the same value in $c$.

If $q > 4$ is odd, we get the following variant of equation \eqref{EqCase12}.
\[
 | \supp(c) \cap \kappa' | = x_1 - x_2 \geq \frac{q-1}{q} \theta_k - q^{k-1} > \frac{1}{2}\theta_k.
\]
The last inequality holds as $q > 4$.
This results yet again in a contradiction by Observation \ref{ObsContradiction}.\\

\underline{Case 2: Possibility (P2) does occur.}

Take a hyperplane $\pi$ and a point $R \not \in \pi \cup \supp(c)$ such that $\pr R \pi c = \alpha_1 \kappa_1 + \alpha_2 \kappa_2$ for some $\alpha_i \in \fp^*$ and distinct $k$-spaces $\kappa_i$ of $\pi$.
Define the following notation:
\begin{align*}
    \sigma := \kappa_1 \cap \kappa_2, &&
    s := \dim(\sigma), &&
    \tau := \langle R, \sigma \rangle, &&
    \lambda_i := \langle R, \kappa_i \rangle.
\end{align*}

Remark that, if $q \geq 5$, Lemma \ref{LmAtLeastHalfPointsSameValue} implies a contradiction to Observation \ref{ObsContradiction}.
As such, we can assume that $q \leq 4$, which implies that $W(k,q) = 2q^k$.

First, remark that $\supp(c) \subseteq \lambda_1 \cup \lambda_2$.
Indeed, as $\wt(c) \leq 2q^k$ and $s \leq k-1$, we know that $\lambda_1 \cup \lambda_2$ contains at least $2(\theta_k - \theta_{k-1}) = 2q^k$ points of $\supp(c)$.
This is only possible if $\wt(c) = 2q^k$ and thus $\supp(c) \subseteq \lambda_1 \cup \lambda_2$.
Note that this means that $\pr R \pi c = \alpha_1 (\kappa_1 - \kappa_2)$, and $s = k-1$.

Now take a point $Q \in \lambda_1 \setminus (\lambda_2 \cup \supp(c))$.
We can assume, w.l.o.g., that $Q \not \in \pi$ (else, by remark \ref{RmkProjHyperplane}, we choose another hyperplane $\pi$).
Then $Q$ projects every point of $\lambda_1$ to a point of $\kappa_1$, and for every point $P$ of $\lambda_2 \setminus \tau$, $Q P$ can't contain a point of $\supp(c)$ other than $P$.
Hence, the points of $(\lambda_2 \setminus \tau) \cap \supp(c)$ are projected by $Q$ onto points with non-zero value in $\pr Q \pi c$.
In particular, $\pr Q \pi c \neq \zero$.
By Lemma \ref{LmProj} (5), this implies that $\pr Q \pi c$ is a linear combination of precisely two $k$-spaces.
Even more, as $\wt(c) = 2q^k$, we know that $\pr Q \pi c$ is the difference of two distinct $k$-spaces through a $(k-1)$-space.

The fact that $\wt( \pr Q \pi c) = 2q^k$ is only possible if no line through $Q$ contains more than one point of $\supp(c)$.
In this way, we see that all points of $\kappa_1 \setminus \sigma$ must have value $\alpha_1$ in $\pr Q \pi c$.
Thus, $\pr Q \pi c = \alpha_1 (\kappa_1 - \rho)$ for some $k$-space $\rho$ in $\pi$.\footnote{
Beware that if $q=2$ and $c = \kappa_1 + \kappa_2$, with $\kappa_1$ and $\kappa_2$ $k$-spaces through a $(k-1)$-space, these spaces $\kappa_1$ and $\kappa_2$ are not uniquely determined by $c$.
This is because, if $K = \vspan{\kappa_1, \kappa_2}$, then $K \setminus \supp(c)$ is a $k$-space $\kappa_3$.
If $\kappa'_1$ and $\kappa'_2$ are distinct $k$-spaces in $K$, intersecting $\kappa_3$ in the same $(k-1)$-space, then also $c = \kappa'_1 + \kappa'_2$.}
This means that all points of $\supp(c) \cap (\lambda_2 \setminus \tau)$ have value $-\alpha_1$ and lie in the space $\mu := \lambda_2 \cap \langle Q, \rho \rangle$.
Note that $\dim(\mu) \leq k$ and $\mu$ contains $q^k > \frac{1}{2}\theta_k$ points of $\supp(c)$ with value $-\alpha_1$ in $c$.
Observation \ref{ObsContradiction} yields the desired contradiction.\\

If $q \in Q_3 \cup Q_4 \cup Q_5$, then the bound is tight because it is tight for $\mc_k(k+1,q)$ (see Result \ref{ResLins}) and we can interpret $\mc_k(k+1,q)$ as a subcode of $\mc_k(n,q)$ by restricting the generating set $G_k^{(0)}(n,q)$ of $\mc_k(n,q)$ to $G_k^{(0)}(\Pi)$ for some $(k+1)$-space $\Pi$ in $\pg(n,q)$.
This way we see that $\mc_{j,k}(n,q)$ must also contain code words of weight $W(k,q)+1$.
Note that $W(k,q)+1$ exceeds $2 \theta_k$, which is an upper bound on the weight of a linear combination of two $k$-spaces.
\end{proof}

\setcounter{obs}{0}

\begin{crl}
 \label{CrlMinWtHullPoints}
If $c$ is a code word of $\mh_{0,k}(n,q)$, with $\wt(c) \leq W(k,q)$, then $c$ is a scalar multiple of the difference of two $k$-spaces. In particular, the minimum weight of $\mh_{0,k}(n,q)$ is $2q^k$, and the minimum weight code words are scalar multiples of the difference of two $k$-spaces through a common $(k-1)$-subspace.
\end{crl}

\begin{proof}
The arguments are the same as in Step 3 of the proof of Theorem \ref{ThmLargeQMainThm}.
\end{proof}

\begin{rmk}\label{RmkWeightSpec}
It is not difficult to write down the weight spectrum of $\mc_k(n,q)$ explicitly for weights up to $W(k,q)$.
For all $q$, the minimum weight code words have weight $\theta_k$ and are the scalar multiples of $k$-spaces.
The next weight is $2q^k$ and is attained only by the scalar multiples of the difference of two $k$-spaces intersecting in a $(k-1)$-space.
In general, if $\alpha_1, \alpha_2 \in \fp^*$ and $\kappa_1, \kappa_2 \in G_k$ with $\kappa_1 \neq \kappa_2$, then $\wt(\alpha_1 \kappa_1 + \alpha_2 \kappa_2) = 2\theta_k - (1 + \eps) \theta_{\dim(\kappa_1 \cap \kappa_2)}$, with $\eps = 1$ if $\alpha_1 = - \alpha_2$, and $\eps = 0$ otherwise.

In particular, we know that $[2\theta_k - \theta_{2k-n} + 1, W(k,q)]$ is a gap in the weight spectrum.
This interval in non-empty if $q \notin Q_1$ and if either $q \notin Q_2$ or $2k \geq n$.
\end{rmk}

\section{Codes of $\boldsymbol{j}$- and $\boldsymbol{k}$-spaces}
\label{SectJKSpaces}
The main goal of this section is generalising Theorem \ref{ThmPointsKSpaces} to all codes $\mc_{j,k}(n,q)$.
The following map, which is essentially due to Bagchi \& Inamdar \cite{bagchi}, will prove to be very helpful.\footnote{In this section, we will denote two distinct projections with Devanagari symbols.
These can be imported in \LaTeX\, using the package \texttt{devanagari}.
In Definition \ref{DefPa}, we introduce the symbol $\pa$ (pronounced `pa' with corresponding command \verb!{\dn p}!), while, in Definition \ref{DefLa}, we use the symbol $\la$ (pronounced `la' with corresponding command \verb!{\dn l}!).}

\begin{df}
\label{DefPa}
Looking at $V(j,n,q)$, the elements of $G_j^{(j)}$ form the standard basis.
Given an $i$-space $\iota$ of $\pg(n,q)$, with $-1 \leq i < j$, we take an $(n-i-1)$-space $\pi$ of $\pg(n,q)$, skew to $\iota$.
Consider the unique linear map $\pa_\iota: V(j,n,q) \rightarrow V(j-i-1,\pi)$ satisfying, for all $\lambda \in G_j^{(j)}$,
\[
 \pa_\iota(\lambda) = \begin{cases}
 \lambda \cap \pi & \text{if } \iota \subset \lambda,\\
 \zero & \text{otherwise}.
 \end{cases}
\]
This means that, given $v \in V(j,n,q)$ and a $(j-i-1)$-space $\mu \subset \pi$, we have $\pa_\iota(v)(\mu) = v(\vspan{\mu,\iota})$.
\end{df}

Note that $\pa_\iota$ is closely related to taking the quotient of $\pg(n,q)$ through the space $\iota$.
The choice of $\pi$ doesn't make a (qualitative) difference for the definition of $\pa_\iota$.

\begin{lm}[{\cite[Theorem 1]{bagchi}}]\label{LmJSpacesThroughPoint}
Assume that $c \in \mc_{j,k}(n,q)$, with $j \geq 1$, and let $\iota$ be an $i$-space of $\pg(n,q)$, with $-1 \leq i < j$.
Then $\pa_\iota(c) \in \mc_{j-i-1,k-i-1}(n-i-1,q)$.
\end{lm}

\begin{proof}
Take a $\kappa \in G_k^{(j)}$.
It is easy to see that
\[
 \pa_\iota(\kappa) = \begin{cases}
  \kappa \cap \pi & \text{if } \iota \subset \kappa,\\
  \zero & \text{otherwise},
 \end{cases}
\]
which implies that the image of $G_k(n,q)^{(j)}$ under $\pa_\iota$ is $G_{k-i-1}(\pi)^{(j)} \cup \set{\zero}$.
These sets generate $\mc_{j,k}(n,q)$ and $\mc_{j-i-1,k-i-1}(n-i-1,q)$, respectively.
Hence, it follows that $\pa_\iota \big( \mc_{j,k}(n,q) \big) = \mc_{j-i-1,k-i-1}(n-i-1,q)$.
\end{proof}

Another map that will serve as a useful tool is the following.

\begin{df}\label{DefLa}
Take an integer $i$, with $0 \leq i < j$.
Define the map:
\[
 \la_i: V(j,n,q) \rightarrow V(i,n,q):
 v \mapsto \bigg(\la_i(v): \iota \mapsto \sum_{\iota \subset \lambda \in G_j} v(\lambda)\bigg).
\]
This means that the value of $\la_i(v)$ at an $i$-space $\iota$ is the sum of the values in $v$ of all $j$-spaces $\lambda$ through $\iota$.
We will denote $\la_0$ by $\la$.
\end{df}

\begin{lm}
 \label{LmLaLin}
The map $\la_i$ is linear and $\la_i\big(\mc_{j,k}(n,q)\big) = \mc_{i,k}(n,q)$.
\end{lm}

\begin{proof}
Take $\alpha, \beta \in \fp$ and $v,w \in V(j,n,q)$.
Let $\iota$ be an $i$-space of $\pg(n,q)$.
Then
\begin{align*}
 \la_i(\alpha v + \beta w)(\iota)
 &= \sum_{\iota \subset \lambda \in G_j} (\alpha v + \beta w)(\lambda)
 = \sum_{\iota \subset \lambda \in G_j} (\alpha v(\lambda) + \beta w(\lambda))\\
 &= \alpha \sum_{\iota \subset \lambda \in G_j}  v(\lambda) + \beta \sum_{\iota \subset \lambda \in G_j} w(\lambda))
 = \alpha\, \la_i(v)(\iota) + \beta\, \la_i(w)(\iota).
\end{align*}
Since this holds for every $i$-space $\iota$, $\la_i(\alpha v + \beta w) = \alpha\, \la_i(v) + \beta\, \la_i(w)$.

Now take a $k$-space $\kappa$ and an $i$-space $\iota$.
\[
 \la_i(\kappa^{(j)})(\iota) 
 = \sum_{\iota \subset \lambda \in G_j} \kappa^{(j)}(\lambda)
 = \sum_{\substack{\lambda \in G_j \\ \iota \subset \lambda \subset \kappa}} 1
 = \begin{cases}
  \gauss {k-i}{j-i} \equiv 1 \pmod p & \text{if } \iota \subset \kappa,\\
  0 & \text{otherwise,}
 \end{cases}
 = \kappa^{(i)}(\iota).
\]
This means that $\la_i(\kappa^{(j)}) = \kappa^{(i)}$.
Hence, the generators of $\mc_{j,k}(n,q)$ are mapped to the generators of $\mc_{i,k}(n,q)$.
Since $\la_i$ is linear, this proves that $\la_i \big( \mc_{j,k}(n,q) \big) = \mc_{i,k}(n,q)$.
\end{proof}

\begin{lm}
 \label{LmCompLa}
Assume that $v \in V(j,n,q)$ and $0 \leq i < j$.
Then $\la\big(\la_i(v)\big) = \la(v)$.
\end{lm}

\begin{proof}
Take an arbitrary point $P$ in $\pg(n,q)$.
We need to prove that $\la\big(\la_i(v)\big)(P) = \la(v)(P)$.
\begin{align*}
\la\big(\la_i(v)\big)(P) 
& = \sum_{P \in \iota \in G_i} \la_i(v)(\iota)
= \sum_{P \in \iota \in G_i} \sum_{\iota \subset \lambda \in G_j} v(\lambda)
= \sum_{P \in \lambda \in G_j} v(\lambda) \Big( \sum_{\substack{\iota \in G_i \\ P \in \iota \subset \lambda}} 1 \Big)\\
& = \sum_{P \in \lambda \in G_j} v(\lambda) \gauss{j}{i}
\equiv \sum_{P \in \lambda \in G_j} v(\lambda)
= \la(v)(P) \pmod p . \qedhere
\end{align*}
\end{proof}

The following lemma shows the interaction between $\pa$ and $\la$.

\begin{lm}
 \label{LmCoeffLa}
Assume that $c \in \mc_{j,k}(n,q)$, and let $\iota$ be an $i$-space, with $0 \leq i < j$.
Then $\la_i(c)(\iota) = \pa_\iota(c) \cdot \one$.
Hence, $\la_i(c)(\iota) = 0$ if and only if $\pa_\iota(c) \in \mh_{j-i-1,k-i-1}(n-i-1,q)$.
\end{lm}

\begin{proof}
It is easy to see that both $\la_i(c)(\iota)$ and $\pa_\iota(c) \cdot \one$ equal the sum of the values in $c$ of all $j$-spaces through $\iota$.
We know that $\pa_\iota(c) \in \mc_{j-i-1,k-i-1}(n-i-1,q)$.
By Lemma \ref{LmHull} (2), this means that $\pa_\iota(c) \in \mh_{j-i-1,k-i-1}(n-i-1,q)$ if and only if $\pa_\iota(c) \cdot \one = 0$.
\end{proof}

We can now characterise all code words of $\mc_{j,k}(n,q)$ up to weight $W(j,k,q)$.
If $q$ is large enough, then this bound exceeds $2\gauss{k+1}{j+1}$, which is at least the maximum weight of a linear combination of two $k$-spaces (with equality if and only if $n > 2k - j$).

\begin{thm}
 \label{ThmLargeQMainThm}
Assume that $q \notin Q_1$.
\begin{enumerate}
 \item[(1)] If $c$ is a code word of $\mc_{j,k}(n,q)$, with $\wt(c) \leq W(j,k,q)$, then $c$ is a linear combination of at most two $k$-spaces.
 \item[(2)] If $c$ is a code word of $\mh_{j,k}(n,q)$, with $\wt(c) \leq W(j,k,q)$, then $c$ is a scalar multiple of the difference of two $k$-spaces.
 In particular, the minimum weight of $\mh_{j,k}(n,q)$ is $2q^{k-j} \gauss{k}{j}$, and the minimum weight code words are scalar multiples of the difference of two $k$-spaces through a common $(k-1)$-space.
\end{enumerate}
\end{thm}

\begin{proof}
We will prove this by induction on $j$.
If $j = 0$, this follows from Theorem \ref{ThmPointsKSpaces} and Corollary \ref{CrlMinWtHullPoints}, as $W(0,k,q) \leq W(k,q)$.
So assume that $j \geq 1$ and that the theorem holds for all codes $\mc_{j',k'}(n',q)$, with $j' < j$, and $j' < k' < n'$.\\

\underline{Step 1: Attain a lower bound on the minimum weight of $\ker(\la_{j-1}) \cap \mc_{j,k}(n,q)$.}

Let $c$ be a non-zero code word of $\mc_{j,k}(n,q)$, with $\la_{j-1}(c) = \zero$.
We will find a lower bound on $\wt(c)$ by performing a double count on the set
\[
 S := \set{(P,\lambda) : P \in \supp_0(c), \, P \in \lambda \in \supp(c)}.
\]

We know that $c \neq \zero$ means that $\supp(c) \neq \emptyset$, hence $\supp_{j-1}(c) \neq \emptyset$.
Take a subspace $\iota \in \supp_{j-1}(c)$.
It follows from Lemma \ref{LmCoeffLa} that $\pa_{\iota}(c) \in \mh_{0,k-j}(n-j,q)$.
Recall that $\wt(\pa_\iota(c))$ equals the number of $j$-spaces of $\supp(c)$ through $\iota$.
Since $\iota \in \supp_{j-1}(c)$, this number is not zero.
Therefore, $\pa_\iota(c)$ is a non-zero code word of $\mh_{0,k-j}(n-j,q)$.
Thus, by Corollary \ref{CrlMinWtHullPoints}, we have that $\wt(\pa_\iota(c)) \geq 2q^{k-j}$.
Hence, $\supp(c)$ contains at least $2q^{k-j}$ $j$-spaces through $\iota$.
This yields that
\[
 | \supp_0(c) | \geq \theta_{j-1} + 2q^{k-j}(\theta_j - \theta_{j-1}) > 2q^k.
\]
Now take a point $P \in \supp_0(c)$.
On the one hand, Lemma \ref{LmCompLa} assures us that $\la(c)(P) = \la(\la_{j-1}(c))(P) = \la(\zero)(P) = 0$.
Lemma $\ref{LmCoeffLa}$ then implies that $\pa_P(c) \in \mh_{j-1,k-1}(n-1,q)$.
On the other hand, $P \in \supp_0(c)$, so $\pa_P(c) \neq \zero.$
Using the induction hypothesis, we get $\wt(\pa_P(c)) \geq 2q^{k-j}\gauss{k-1}{j-1}$.
Thus, the number of  $j$-spaces of $\supp(c)$ through $P$ is at least $2q^{k-j}\gauss{k-1}{j-1}$.
This yields that
\[
 \wt(c) \theta_j = |S| \geq |\supp_0(c)| \cdot  2q^{k-j}\gauss{k-1}{j-1} > 4 q^{2k-j}\gauss{k-1}{j-1}.
\]

One can check that 
\begin{align*}
\frac{q^k}{\theta_j} > \left(1-\frac{1}{q}\right) \frac{q^{k+1}-1}{q^{j+1}-1}
& & \text{and} & &
q^{k-j} > \left( 1- \frac1q \right)\frac{q^k-1}{q^j-1}.
\end{align*}
Therefore, if we take into account that $q \geq 11$, the above inequalities imply that
\[
 \wt(c) > 4 \frac{q^k}{\theta_j} q^{k-j} \gauss{k-1}{j-1}
 > 4 \left(1-\frac{1}{11}\right)^2 \frac{q^{k+1}-1}{q^{j+1}-1} \frac{q^k-1}{q^j-1} \gauss{k-1}{j-1} > 3.3 \gauss{k+1}{j+1}
\]

Note that, in particular, $\wt(c) > W(j,k,q)$.\\

\underline{Step 2: Applying this lower bound to characterise low weight code words.}

Assume that $c$ is a code word of $\mc_{j,k}(n,q)$, with $\wt(c) \leq W(j,k,q)$.
Now, double count the set
\[
 S := \sett{(\iota,\lambda)}{\iota \in \supp_{j-1}(c), \, \iota \subset \lambda \in \supp(c)}.
\]
We know that if $\iota \in \supp_{j-1}(c)$, then $\pa_\iota(c)$ is a non-zero code word of $\mc_{0,k-j}(n-j,q)$.
Therefore, $\wt(\pa_\iota(c)) \geq \theta_{k-j}$.
Note that $\wt(\pa_\iota(c))$ equals the number of $j$-spaces $\lambda \in \supp(c)$ through $\iota$.
Also note that $\supp(\la_{j-1}(c)) \subseteq \supp_{j-1}(c)$.
This yields
\[
 \wt(c) \theta_j 
 = |S|
 = \sum_{\iota \in \supp_{j-1}(c)} \wt(\pa_\iota(c))
 \geq \wt(\la_{j-1}(c)) \theta_{k-j}.
\]
This means that
\[
 \wt(\la_{j-1}(c)) \leq \frac{\theta_j}{\theta_{k-j}} \wt(c)
 \leq \frac{\theta_j}{\theta_{k-j}} W(j,k,q) = W(j-1,k,q).
\]

The last inequality relies on the fact that $\frac{\theta_j}{\theta_{k-j}}\gauss{k+1}{j+1} = \gauss{k+1}{j}$.

The induction hypothesis tells us that $\la_{j-1}(c)$ is a linear combination of at most two $k$-spaces.
Thus, $\la_{j-1}(c) = \alpha \kappa_1^{(j-1)} + \beta \kappa_2^{(j-1)}$, for some $\alpha, \beta \in \fp$, and $\kappa_i \in G_k$.
Note that $\alpha$ or $\beta$ can be zero.

Now assume that $c \neq \alpha \kappa_1^{(j)} + \beta \kappa_2^{(j)}$.
If $\supp(c) \subseteq G_j(\kappa_1) \cup G_j(\kappa_2)$, then $\supp(c - \alpha \kappa_1^{(j)} - \beta \kappa_2^{(j)}) \subseteq G_j(\kappa_1) \cup G_j(\kappa_2)$, which would mean that $c - \alpha \kappa_1 - \beta \kappa_2$ were a non-zero code word of $\ker(\la_{j-1}) \cap \mc_{j,k}(n,q)$ of weight at most $2 \gauss{k+1}{j+1}$, contradicting Step 1.

Therefore, there exists a $j$-space $\lambda \in \supp(c)$, with $\lambda \not \subset \kappa_1 \cup \kappa_2$.
Hence, we can choose a $(j-1)$-space $\iota \subset \lambda$, which is not entirely contained in $\kappa_1 \cup \kappa_2$.
This means that $\la_{j-1}(c)(\iota) = \alpha \kappa_1(\iota) + \beta \kappa_2(\iota) = 0$.
Since $\iota \in \supp_{j-1}(c)$, this means that $\wt(\pa_\iota(c)) \geq 2q^{k-j}$.
Hence, we find at least $2q^{k-j}$ $j$-spaces of $\supp(c)$ through $\iota$.
Note that all these $j$-spaces contain at least $\theta_j - 3 \theta_{j-1} = q^j - 2\theta_{j-1}$ points $P$ outside of $\iota$, $\kappa_1$ and $\kappa_2$.
Every such point $P$ lies in a unique $j$-space through $\iota$, hence there at least $2q^{k-j} (q^j-2\theta_{j-1})$ points in $\supp_0(c)$, outside of $\kappa_1 \cup \kappa_2$.
Since these points have value zero in $\la(c)$, they lie in at least $2q^{k-j} \gauss{k-1}{j-1}$ $j$-spaces of $\supp(c)$.
As in Step 1, we obtain

\[
 \wt(c) \theta_j \geq 2 q^{k-j} \underbrace{(q^j - 2\theta_{j-1})}_{ > q^j \frac{q-3}{q-1}} 2q^{k-j} \gauss{k-1}{j-1}
 > 4q^{2k-j} \frac{q-3}{q-1} \gauss{k-1}{j-1}.
\]

Therefore,

\[
 \wt(c) \geq 4 \left( 1 - \frac 1 q \right)^2 \frac{q-3}{q-1} \gauss{k+1}{j+1}
 > \left( 4 - \frac{16}q \right) \gauss{k+1}{j+1}
 > W(j,k,q),
\]
 a contradiction.
Hence, $c = \alpha \kappa_1^{(j)} + \beta \kappa_2^{(j)}$.\\

\underline{Step 3: The minimum weight of $\mh_{j,k}(n,q)$.}

The previous characterisation teaches us that the only code words of $\mh_{j,k}(n,q)$ of weight at most $W(j,k,q) \geq 2 \gauss{k+1}{j+1}$ are linear combinations of at most two $k$-spaces.
Take such a non-zero code word $c = \alpha \kappa_1 + \beta \kappa_2$.
Then $\alpha + \beta = c \cdot \one = 0$, due to Lemma \ref{LmHull}.
Since $\alpha$ and $\beta$ can't both be zero (then $c$ would be $\zero$), neither of them can be zero.
Write $s = \dim(\kappa_1 \cap \kappa_2)$, then $\wt(c) = 2\gauss{k+1}{j+1} - 2\gauss{s+1}{j+1}$.
This is minimal if $s$ is maximal.
Since $\kappa_1$ and $\kappa_2$ can't coincide (else $c$ would be $\zero$), the maximal value of $s$ is $k-1$.
This yields as minimum weight of $\mh_{j,k}(n,q)$
\[
 2\gauss{k+1}{j+1} - 2\gauss{k}{j+1} = 2q^{k-j} \gauss kj,
\]
and as minimum weight code words the scalar multiples of the difference of two distinct $k$-spaces through a $(k-1)$-space.
\end{proof}

The minimum weight of $\mh_{j,k}(n,q)$ has been an open problem for some time \cite[Open Problem 4.18]{lavrauwC}.
We have solved this problem for $j=0$ in Theorem \ref{ThmPointsKSpaces} and for general $j$ and sufficiently large $q$ in Theorem \ref{ThmLargeQMainThm}.
For smaller values of $q$, we can adapt the arguments to obtain the following weaker statement.

\begin{thm}
\label{ThmMainThmSmallQ}
 If $c$ is a code word of $\mc_{j,k}(n,q)$, with
 \[
 \wt(c) \leq \frac{2q^k}{\theta_j} \gauss k j,
 \]
 then $c = \alpha \kappa$, for some $\alpha \in \fp$, and $\kappa \in G_k$.
 As a consequence, the minimum weight of $\mh_{j,k}(n,q)$ is larger than $2q^k \gauss k j / \theta_j$.
\end{thm}

\begin{proof}
The arguments are essentially the same as the ones used in the proof of Theorem \ref{ThmLargeQMainThm}, so we'll be brief.
Assume that $c$ is a non-zero code word of $\mc_{j,k}(n,q)$ with $\wt(c) \leq \frac{2q^k}{\theta_j}\gauss k j$ and the theorem holds for all smaller values of $j$.

Step 1: Assume that $\pa_{j-1}(c) = \zero$. Double count the set $S$ as in Step 1 above.
We obtain $\wt(c) \geq \frac{2q^k+\theta_{j-1}}{\theta_j} \frac{2q^{k-1}}{\theta_{j-1}}\gauss{k-1}{j-1} > \frac{2q^k}{\theta_j}\gauss k j$, a contradiction.

Step 2: Here we have, similar to the above proof,
\[
 \wt(\pa_{j-1}(c))
 \leq \frac{\theta_j}{\theta_{k-j}} \wt(c)
 \leq \frac{\theta_j}{\theta_{k-j}} \frac{2q^k}{\theta_j} \gauss k j
 = \frac{2q^k}{\theta_{k-j}} \frac{\theta_{k-j}}{\theta_{j-1}} \gauss k {j-1}
 = \frac{2q^k}{\theta_{j-1}} \gauss k {j-1}.
\]
Therefore, the induction hypothesis implies that $\la_{j-1}(c) = \alpha \kappa$ for some scalar $\alpha \in \fp^*$ and a $k$-space $\kappa$.
As above, if $c \neq \alpha \kappa$, then $\supp(c) \not \subseteq G_j(\kappa)$.
Thus, there exists a $(j-1)$-space $\iota \in \supp_{j-1}(c)$ with $\la_{j-1}(\iota) = 0$.
Then $\pa_\iota(c)$ is a non-zero code word of $\mh_{k-j}(n-j,q)$ and we know that $\supp_0(c) \geq 2q^k + \theta_{j-1}$.
Hence, $\wt(c) \theta_j \geq (2q^k + \theta_{j-1}) \gauss k j$, a contradiction.

Step 3: No scalar multiple of a $k$-space is a non-zero code word of $\mh_{j,k}(n,q)$.
\end{proof}

The authors expect that Theorem \ref{ThmLargeQMainThm} (2) holds for all values of $q$.
For instance, Theorem \ref{ThmLargeQMainThm} (1) can be proven to hold for $\mc_{1,2}(n,q)$, $q \neq 2$ up to weight $2 \theta_2$, which proves (2) for $\mh_{1,2}(n,q)$, $q \neq 2$.\\

As we have done in Remark \ref{RmkWeightSpec}, one can now study the weight spectrum of $\mc_{j,k}(n,q)$ up till weight $W(j,k,q)$ using Theorem \ref{ThmLargeQMainThm} and \ref{ThmMainThmSmallQ}.

\subsection*{The cyclicity of $\mc_{j,k}(n,q)$}

A natural question to ask is whether the codes $\mc_{j,k}(n,q)$ are cyclic.
A code $C$, where the code words are denoted as vectors, is \emph{cyclic} if for each code word $(c_1,\dots,c_n) \in C$ its \emph{right shift} $(c_n,c_1,c_2,\dots,c_{n-1})$ is also a code word of $C$.

It has been known for a long time that the codes $\mc_k(n,q)$ are cyclic, see e.g.\ \cite{delsarte}.
Denote $g := \gauss{n+1}{j+1}$.
Then $\mc_{j,k}(n,q)$ is equivalent to a cyclic code if and only if the following holds:
there exists some ordering on the $j$-spaces of $\pg(n,q)$ (write $G_j(n,q) = \set{ \lambda_1, \lambda_2, \dots, \lambda_g }$ and let $\lambda_0$ be equal to $\lambda_g$) such that if $c \in \mc_{j,k}(n,q)$, then $R(c) \in \mc_{j,k}(n,q)$ as well, with $R(c)(\lambda_i) := c(\lambda_{i-1})$.

Given a $k$-space $\kappa$, this would mean that $R(\kappa)$ is also a code word of $\mc_{j,k}(n,q)$.
Furthermore, it's easy to see that $\wt(R(\kappa)) = \wt(\kappa) = \gauss{k+1}{j+1}$, and that $R(\kappa)$ only takes the values 0 and 1.
By Result \ref{ResLowestWeight}, this means that $R(\kappa) = \kappa'$ for some $k$-space $\kappa'$.

This means that the map $f: G_j \rightarrow G_j: \lambda_i \mapsto \lambda_{i-1}$ maps the $j$-spaces in a certain $k$-space to the $j$-spaces of another $k$-space.
But then $f$ can be extended to a collineation on all subspaces of $\pg(n,q)$.
Note that $f$ works cyclically on the $j$-spaces, meaning that the permutation group generated by $f$ has a unique orbit when viewed as permutation group of $G_j$.

Conversely, if such a collineation $f$ exists, we can choose a $\lambda \in G_j$ and write $\lambda_1 = \lambda$, and $\lambda_{i+1} = f(\lambda_i)$.
Under this ordering of the $j$-spaces, $\mc_{j,k}(n,q)$ is cyclic.
This yields the following statement:

\begin{obs}
 \label{ObsCyclic}
The code $\mc_{j,k}(n,q)$ is equivalent to a cyclic code if and only if there exists a collineation $f$ of $\pg(n,q)$, working cyclically on the $j$-spaces.
\end{obs}

It is folklore under finite geometers that the collineations with largest order are Singer cycles, which act cyclically on the points and hyperplanes.
However, a reference is hard to find.
We will use a similar (but in this context weaker) result that suits our purpose.

\begin{res}[{\cite[Corollary 2]{darafsheh}}]\label{ResOrderCollineation}
The maximal order of an element of $\textnormal{GL}(n,q)$ is $q^{n}-1$.
\end{res}

This leads to the following Theorem.

\begin{thm}
 \label{ThmCyclic}
The code $\mc_{j,k}(n,q)$ is equivalent to a cyclic code if and only if $j = 0$.
\end{thm}

\begin{proof}
In the codes we consider, we have the restriction $0 \leq j < k < n$.
By Observation \ref{ObsCyclic}, we need to prove that some collineations work cyclically on the points, but no collineation works cyclically on the $j$-spaces if $0 < j < n-1$.
It is known that Singer cycles are collineations working cyclically on the points and hyperplanes of $\pg(n,q)$, and that such collineations exist for any Desarguesian projective space.
Hence, this proves that $\mc_{k}(n,q)$ is equivalent to a cyclic code.

Now assume that $1 \leq j \leq n-2$.
Let $f$ be a collineation on $\pg(n,q)$.
The Fundamental Theorem of projective geometry teaches us that $f \in \text{P$\Gamma$L}(n+1,q)$.
This is a quotient group of $\Gamma \text{L}(n+1,q)$, which is a subgroup of $\text{GL}((n+1)h,p)$.
Therefore, the order of $f$ cannot exceed the maximal order of an element of $\text{GL}((n+1)h,p)$, which is $p^{(n+1)h} - 1 = q^{n+1}-1$, by Result \ref{ResOrderCollineation}.
But if $f$ would work cyclically on the $j$-spaces of $\pg(n,q)$, then its order would be a multiple of $\gauss{n+1}{j+1}$, which exceeds $q^{n+1}-1$ if $n \geq 3$ and $1 \leq j \leq n-2$.
This contradiction concludes the proof.
\end{proof}

\section{Minimum weight of the dual code}
\label{SectDual}

Throughout \cite{lins} and Section \ref{SectPointsKSpaces} and \ref{SectJKSpaces}, we characterise small weight code words of $\mc_{j,k}(n,q)$ by starting from $\mc_{0,1}(2,q)$ and using induction to generalise the results.
Unfortunately, it is not possible to do something similar for the dual code.
The problem of determining the minimum weight of $\mc_{0,1}(2,q)^\perp$, and characterising minimum weight code words, is still open in general.
However, we can work in opposite direction, and reduce the minimum weight problem of $\mc_{j,k}(n,q)^\perp$ to the codes $\mc_{0,1}(n,q)^\perp$.
A construction by Bagchi \& Inamdar is key.

\begin{constr}[{\cite[Lemma 4]{bagchi}}]\label{ConstrPullBack}
Consider the code $\mc_{j,k}(n,q)^\perp$.
Take a $(j-1)$-space $\iota$, and an $(n-j)$-space $\pi$, skew to $\iota$.
Let $\pi$ play the role of $\pg(n-j,q)$, and let $c$ be a code word of $\mc_{k-j}(n-j,q)^\perp$.
Define $c^+_\iota \in V(j,n,q)$ as
\[
 c^+_\iota(\lambda) := 
 \begin{cases}
  c(\lambda \cap \pi) & \text{if } \iota \subset \lambda, \\
  0 & \text{otherwise.}
 \end{cases}
\]
Then $c_\iota^+ \in \mc_{j,k}(n,q)^\perp$ and $\wt(c_\iota^+) = \wt(c)$.
Code words of this form are called \emph{pull-backs}.
\end{constr}

\begin{proof}
A $j$-space $\lambda$ lies in $\supp(c_\iota^+)$ if and only if $\lambda$ contains $\iota$, and intersects $\pi$ in a point of $\supp(c)$.
Since every point of $\pi$ lies in a unique $j$-space through $\iota$, we get $\wt(c_\iota^+) = \wt(c)$.
Now take a $k$-space $\kappa$.
If $\iota \not \subset \kappa$, then $\kappa$ contains no $j$-spaces of $\supp(c_\iota^+)$, hence $\kappa \cdot c_\iota^+ = 0$.
If $\iota \subset \kappa$, then it easy to see that $\kappa \cdot c_\iota^+ = (\kappa \cap \pi) \cdot c = 0$.
The last equality holds because $\kappa$ intersects $\pi$ in a $(k-j)$-space, and $c \in \mc_{k-j}(n-j,q)^\perp$.
\end{proof}

\begin{rmk}
 \label{RmkPullBack}
A code word $c \in \mc_{j,k}(n,q)$ is a pull-back if and only if all $j$-spaces of $\supp(c)$ go through the same $(j-1)$-space $\iota$.
If the latter holds, then $\pa_\iota(c) \in \mc_{k-j}(n-j,q)^\perp$, and $c = (\pa_\iota (c))_\iota^+$.
\end{rmk}

The previous remark asserts that the standard words of $\mc_{j,k}(n,q)^\perp$ (see Definition \ref{DefStandardWord}) are pull-backs if $j > 0$.
In fact, they are pull-backs of standard words of $\mc_{k-j}(n-j,q)^\perp$.
Bagchi \& Inamdar \cite[Conjecture]{bagchi} conjectured that the minimum weight code words of $\mc_{j,k}(n,p)^\perp$ are standard words, for $p$ prime.
They proved it for $j = k-1$, see Result \ref{ResMaxMinWtDualCode}, and $q=2$ \cite[Proposition 3]{bagchi}.
They also mention that it can be proven in the case $j=0$, using the theory of \cite{delsarte}.
Lavrauw, Storme \& Van de Voorde \cite[Theorem 12]{lavrauwB} gave a geometric proof for the case $j=0$, using Result \ref{ResMaxMinWtDualCode}.
We give a short, alternative proof.
This requires the following result, which is a slight alteration of the original statement using Lemma \ref{LmHull} (2).

\begin{res}[{\cite[Theorem 5.7.9]{assmus}}]\label{ResPPrimeHullIsDual}
If $p$ is prime, then $\mc_k(n,p)^\perp = \mh_{n-k}(n,p)$.
\end{res}

\begin{crl}
 \label{CrlMinWeightDualPrimePoints}
If $p$ is prime, the minimum weight code words of $\mc_k(n,p)^\perp$ are the scalar multiples of the standard words.
\end{crl}

\begin{proof}
A standard word of $\mc_k(n,p)^\perp$ is the difference of two $(n-k)$-spaces through an $(n-k-1)$-space.
This corollary now follows directly from Corollary \ref{CrlMinWtHullPoints} and Result \ref{ResPPrimeHullIsDual}.
\end{proof}

Putting these considerations together simplifies the conjecture of Bagchi \& Inamdar.
To finish the proof of the conjecture, we need to show that minimum weight code words of $\mc_{j,k}(n,q)^\perp$, $j>0$ and $q$ prime, are pull-backs.
It will turn out $q$ need not even be prime.

\begin{lm}
 \label{LmPullBack}
If $j> 0$, then all code words $c \in \mc_{j,j+1}(n,q)^\perp$, with $\wt(c) < 2 \theta_{n-j-1}$, are pull-backs.
In particular, this applies to the minimum weight code words.
\end{lm}

\begin{proof}
Take a non-zero code word $c \in \mc_{j,j+1}(n,q)^\perp$, with $\wt(c) < 2 \theta_{n-j-1}$.
Take a $(j-1)$-space $\iota$, define $X := \set{ \lambda \in \supp(c) : \iota \subset \lambda}$, and denote $x := |X|$.
Assume that $X \neq \emptyset$.

Take a $j$-space $\lambda_1 \in X$.
Then every other element $\lambda_2$ of $X$ lies is a unique $(j+1)$-space through $\lambda_1$.
Therefore, there are at least $\gauss{n-j}{(j+1)-j} - (x-1) = \theta_{n-j-1} - x + 1$ $(j+1)$-spaces $\kappa$ through $\lambda_1$, not containing another element of $X$.
Each such space $\kappa$ contains another element $\lambda_3$ of $\supp(c) \setminus X$, otherwise $\kappa \cdot c = c(\lambda_1) \neq 0$, contradicting the fact that $c \in \mc_{j,j+1}(n,q)^\perp$.
Note that $\lambda_3$ doesn't lie in a $(j+1)$-space with another element $\lambda_2 \in X \setminus \set{\lambda_1}$.
Otherwise, $\lambda_2$ would intersect $\lambda_1$ in $\iota$ and $\lambda_3$ in another $(j-1)$-space (since $\lambda_3 \not \in X$), which implies that $\lambda_2 \subset \vspan{\lambda_1,\lambda_3} = \kappa$.
This is in contradiction with the way we chose $\kappa$.

Thus, every $\lambda_1 \in X$ gives rise to at least $\theta_{n-j-1} - x + 1$ elements in $\supp(c) \setminus X$, none of which are counted twice.
This yields
\[
 2\theta_{n-j-1} >
 \wt(c) \geq
 x(\theta_{n-j-1} - x + 1 + 1).
\]
This leads to a contradiction for $x=2$ and $x=\theta_{n-j-1}$.
Since the above expression is quadratic in $x$, we can see that it must lead to a contradiction whenever $2 \leq x \leq \theta_{n-j-1}$.

Now take a $j$-space $\lambda_1 \in \supp(c)$ and a $(j+1)$-space $\kappa$ through $\lambda_1$.
As argued above, we know that $\kappa$ must contain another $j$-space $\lambda_2 \in \supp(c)$.
Then $\lambda_1 \cap \lambda_2$ must be some $(j-1)$-space $\iota$.
By the previous arguments, we know that there are at least $\theta_{n-j-1} + 1$ elements of $\supp(c)$ through $\iota$.
Assume that $\lambda$ is an element of $\supp(c)$ not through $\iota$.
Then there is at most one $(j+1)$-space through $\lambda$ containing $\iota$.
This means that there are at least $\theta_{n-j-1} - 1$ $(j+1)$-spaces through $\lambda$, all containing another element of $\supp(c)$ not through $\iota$.
This yields $\wt(c) \geq (\theta_{n-j-1} + 1) + 1 + (\theta_{n-j-1} - 1) > 2\theta_{n-j-1}$, a contradiction.

Therefore, all elements of $\supp(c)$ contain a common $(j-1)$-space $\iota$.
By Remark \ref{RmkPullBack}, this proves that $c$ is a pull-back.
This applies to the minimum weight code words, since the minimum weight of $\mc_{j,j+1}(n,q)$ is at most $2q^{n-j-1}$, see Result \ref{ResMaxMinWtDualCode}.
\end{proof}

The previous lemma was an induction base for the main theorem of this section.
Its proof requires the following construction.

\begin{constr}
 \label{ConstrEmbedding}
 \cite[Theorem 10]{lavrauwB}
Take an $n$-space $\pi$ in $\pg(n+m,q)$ and a code word $c \in \mc_{j,k}(n,q)^\perp \leq V(j,\pi)$. 
Now define $c' \in V(j,n+m,q)$ as
\[
 c'(\lambda) := 
 \begin{cases}
  c(\lambda) & \text{if } \lambda \subset \pi \\
  0 & \text{otherwise}
 \end{cases}.
\]
Then $c' \in \mc_{j,k+m}(n+m,q)$ and $\wt(c') = \wt(c)$.
We call $c'$ an \emph{embedded code word} or a \emph{code word embedded in an $n$-space}.
\end{constr}

\begin{proof}
Take a $(k+m)$-space $\rho$ in $\pg(n+m,q)$.
Then $\rho$ intersects $\pi$ in a space of dimension at least $k$.
As a consequence, we can write $\rho \cap \pi$ (as element of $V(j,\pi)$) as the sum of its $k$-dimensional subspaces.
This yields
\[
 \rho \cdot c'
 = (\rho \cap \pi) \cdot c
 = \left( \sum_{\kappa \in G_k(\rho \cap \pi)} \kappa \right) \cdot c
 = \sum_{\kappa \in G_k(\rho \cap \pi)} (\kappa \cdot c)
 = 0.
\]
Hence, $c' \in \mc_{j,k+m}(n+m,q)^\perp$.
It is trivial that $\wt(c') = \wt(c)$.
\end{proof}

\begin{crl}
 \label{CrlEqMinWt}
\[
 d \left(\mc_{j,k}(n,q)^\perp \right) \geq d \left(\mc_{j,k+1}(n+1,q)^\perp \right).
\]
\end{crl}

\begin{proof}
Take a minimum weight code word $c \in \mc_{j,k}(n,q)^\perp$.
Embedding it in some hyperplane of $PG(n+1,q)$, yields a code word of $\mc_{j,k+1}(n+1,q)^\perp$ of equal weight.
\end{proof}

The proof of the next theorem was inspired by \cite[Section 4]{lavrauwB}.

\begin{thm}
 \label{ThmPullBack}
If $j>0$, then all minimum weight code words of $\mc_{j,k}(n,q)^\perp$ are pull-backs.
\end{thm}

\begin{proof}
Fix a value $j > 0$.
The theorem will be proved through induction on $k$.
We already know it holds for $k = j+1$.
Hence, assume that $k > j+1$, and that the theorem holds for $\mc_{j,k-1}(n-1,q)^\perp$.
Take a minimum weight code word $c \in \mc_{j,k}(n,q)^\perp$.
We know that $\wt(c) \leq 2q^{n-k}$.
Thus,
\[
 | \supp_0(c) |
 \leq \wt(c) \theta_j
 \leq 2q^{n-k} \theta_j.
\]
Take a $j$-space $\lambda \in \supp(c)$.
Assume that every $(j+1)$-space $\rho$ through $\lambda$ contains at least $q^j$ points of $\supp_0(c) \setminus \lambda$.
This yields that
\[
 | \supp_0(c) | 
 \geq \gauss{n-j}{(j+1)-j} q^j + \theta_j
 = \theta_{n-j-1} q^j + \theta_j
 = \theta_{n-1} + q^j.
\]
Putting these inequalities together implies that $2q^{n-k}\theta_j \geq \theta_{n-1} + q^j$, which leads to a contradiction, since $k \geq j+2$.

So take a $(j+1)$-space $\rho$ through $\lambda$ such that $\rho$ contains less than $q^j$ points of $\supp_0(c) \setminus \lambda$.
In particular, this means that $\rho \not \subseteq \supp_0(c)$.
Therefore, there exists a point $R \in \rho \setminus \supp_0(c)$.
If $c \cdot \rho = 0$, then $\rho$ must contain at least one other $j$-space of $\supp(c)$ than $\lambda$, which would also mean that $\rho$ contains at least $q^j$ points of $\supp_0(c) \setminus \lambda$, a contradiction.
Let $\pi$ be a hyperplane not through $R$.
We know from Lemma \ref{LmProj} (3, 4) that $c' := \prj j R \pi c \in \mc_{j,k-1}(n-1,q)^\perp$, and $\wt (c') \leq \wt(c)$.
We also know that $c' (\rho \cap \pi) = c \cdot \rho \neq 0$, so $c' \neq \zero$.

Because $c$ is a minimum weight code word, Corollary \ref{CrlEqMinWt} shows that $\wt(c') = \wt(c)$ and that $c'$ must be a minimum weight code word as well.
Since $\wt(c') = \wt(c)$, Lemma \ref{LmProj} (5) implies that no $(j+1)$-space through $R$ contains more than one $j$-space of $\supp(c)$.

By the induction hypothesis, there exists a $(j-1)$-space $\iota \subset \pi$ contained in all $j$-spaces of $\supp(c')$.
Now take a $j$-space $\lambda \in \supp(c)$.
Then $R$ projects $\lambda$ onto a $j$-space through $\iota$ (note that this holds because $\lambda$ is the only element of $\supp(c)$ in $\vspan{R,\lambda}$, so it gets projected onto an element of $\supp(c)$).
This means that $\vspan{R,\lambda}$ contains $\rho_1 := \vspan{R,\iota}$, hence $\lambda$ intersects $\rho_1$ in a $(j-1)$-space.

Now look at how $R$ was chosen.
We took a $(j+1)$-space $\rho$ through some $\lambda \in \supp(c)$, such that $\rho$ contains less than $q^j$ points of $\supp_0(c) \setminus \lambda$.
Note that $\rho_1$ intersects $\rho$ in at most a $j$-space, hence $\rho_1 \cup \lambda$ contains at most $2q^j +  \theta_{j-1}$ points of $\rho$.
Since $\rho$ contains $\theta_{j+1} \geq 3q^j + \theta_{j-1}$ points, there exists a point $R_2 \in \rho \setminus (\rho_1 \cup \supp_0(c))$.
Take a hyperplane $\pi_2$ not through $R_2$.
Repeating the previous arguments yields again a $j$-space $\rho_2 = \vspan{R_2,\iota_2}$, for some $(j-1)$-space $\iota_2 \subset \pi_2$, such that every $j$-space of $\supp(c)$ intersects $\rho_2$ in $(j-1)$-space.
Note that $R_2 \not \in \rho_1$, so $\rho_1 \neq \rho_2$.

Now take a $j$-space $\lambda \in \supp(c)$.
Then $\rho_1$ and $\rho_2$ both intersect $\lambda$ in a $(j-1)$-space, hence $\dim (\rho_1 \cap \rho_2) \geq \dim ( \rho_1 \cap \rho_2 \cap \lambda ) \geq j-2$.
Assume that $\dim (\rho_1 \cap \rho_2) = j-2$, then $\dim \vspan{\rho_1, \rho_2} = j+2$.
Now every $j$-space $\lambda \in \supp(c)$ intersects $\rho_1$ and $\rho_2$ in a different $(j-1)$-space, thus $\lambda \subset \vspan{\rho_1, \rho_2}$.
This means that $c$ is the embedding of a code word $c' \in \mc_{j,k'}(j+2,q)^\perp$, with $(j+2) - k' = n - k$.
This is only possible if $j < k' < j+2$, hence $k' = j + 1$.
Then $c'$ is a pull-back by Lemma \ref{LmPullBack}.
Thus, $c$ is a pull-back as well.

Now assume that $\dim (\rho_1 \cap \rho_2) = j-1$, and therefore $\dim \vspan{\rho_1, \rho_2} = j+1$.
Furthermore, assume that there exists a $j$-space $\lambda \in \supp(c)$ not through $\rho_1 \cap \rho_2$.
Then $\rho_1$ and $\rho_2$ intersect $\lambda$ in distinct hyperplanes of $\lambda$, hence $\lambda \subset \vspan{\rho_1,\rho_2}$.
Then there exists a $k$-space $\kappa$, intersecting $\vspan{\rho_1,\rho_2}$ in $\lambda$.
Since every $j$-space of $\supp(c)$ either contains $\rho_1 \cap \rho_2$ or is contained in $\vspan{\rho_1, \rho_2}$, this means that $\lambda$ is the only element of $\supp(c)$ contained $\kappa$.
But then $c \cdot \kappa = c(\lambda) \neq 0$, contradicting the fact that $c \in \mc_{j,k}(n,q)^\perp$.
Thus, all $j$-spaces of $\supp(c)$ go through the $(j-1)$-space $\rho_1 \cap \rho_2$.
By Remark \ref{RmkPullBack}, $c$ is a pull-back.
\end{proof}

This reduces the minimum weight problem of $\mc_{j,k}(n,q)^\perp$ to the case $j = 0$.
The following result reduces it further to $k=1$.

\begin{res}[{\cite[Theorem 11]{lavrauwB}}]\label{ResEmbeddedCodeWords} 
Every minimum weight code word of $\mc_k(n,q)^\perp$ is embedded in an $(n-k+1)$-space.
\end{res}

Theorem \ref{ThmPullBack} can generalise some previous work on the codes.

\begin{crl}
 \label{CorGeneralDual}
\begin{enumerate}
 \item[(1)] $d \left( \mc_{j,k}(n,q)^\perp \right) = d \left( \mc_1(n-k+1,q)^\perp \right)$.
 \item[(2)] If $p$ is prime, the minimum weight code words of $\mc_{j,k}(n,p)^\perp$ are scalar multiples of the standard words, and thus have weight $2p^{n-k}$.
 \item[(3)] If $q$ is even, then $d \left ( \mc_{j,k}(n,q)^\perp \right) = (q+2)q^{n-k-1}$.
\end{enumerate}
\end{crl}

\begin{proof}
(1) This follows directly from Theorem \ref{ThmPullBack} and Result \ref{ResEmbeddedCodeWords}. \\
(2) As noted previously, this follows from Corollary \ref{CrlMinWeightDualPrimePoints}, Theorem \ref{ThmPullBack}, and the fact that a pull-back $c_\iota^+$ is a standard word if and only if $c$ is a standard word. \\
(3) This follows from Theorem \ref{ThmPullBack} and Result \ref{ResMaxMinWtDualCodeQeven}.
\end{proof}

If $q$ is odd and not a prime, the minimum weight of $\mc_1(n,q)^\perp$ remains an open problem.
The best known bounds to the authors are the following.

\begin{res}[{\cite[Theorem 3]{bagchi}\cite[Corollary 4.15]{lavrauwC}}]
 If $q$ is odd and not prime, then
 \[
  2q^{n-1} - 2 \frac {q-p} p \theta_{n-2}
  \leq d \left( \mc_1(n,q)^\perp \right)
  \leq 2q^{n-1} - \frac{q-p}{p-1}q^{n-2} .
 \]
\end{res}

There are other interesting constructions.
Small weight code words of $\mc_1(n,q)^\perp$ can be constructed from small weight code words of $\mc_1(2,q)^\perp$.

\begin{constr}
 \label{ConstrTruncCone}
 Let $\pi$ be a plane in $\pg(n,q)$, and take $c \in \mc_1(\pi)^\perp$.
 Let $\tau$ be an $(n-3)$-space, skew to $\pi$.
 Define $c_\tau^- \in V(0,n,q)$ as follows:
 \[
  c_\tau^-(P) = \begin{cases}
   0 & \text{if } P \in \tau, \\
   c(\vspan{P,\tau} \cap \pi) & \text{otherwise}.
  \end{cases}
 \]
 Then $c_\tau^- \in \mc_1(n,q)^\perp$ and $\wt(c_\tau^-) = \wt(c)q^{n-2}$.
\end{constr}

This construction is also described in \cite[Lemma 6]{bagchi}.
Note that $\supp(c_\tau^-)$ is a truncated cone with base $\supp(c)$ and vertex $\tau$.

In \cite{maarten}, subgeometries are used to construct small weight code words.
We can generalise this construction using field reduction.
The idea is as follows (for more details see e.g.\ \cite{lavrauw}).
Choose an exponent $e > 1$.
The projective space $\pg(n,q^e)$ can be recognised in $\pg(N,q)$ with $N = (n+1)e-1$.
The points of $\pg(n,q^e)$ correspond to an $(e-1)$-spread $\mathcal S$ of $\pg(N,q)$.
In general, each $k$-space of $\pg(n,q^e)$ corresponds to a $((k+1)e-1)$-space $\mathcal B(\kappa)$ of $\pg(N,q)$, such that each element of $\mathcal S$ is either skew to $\mathcal B(\kappa)$ or completely contained in $\mathcal B (\kappa)$.

\begin{constr}
\label{ConstrFieldRed}
 Let $e \in \mathbb{N}\setminus\{0,1\}$ and $N := ((n+1)e-1)$.
 Take a code word $c \in \mc_{2e-1}(N,q)^\perp$.
 Define
 \[
  c' : G_0(n,q^e) \rightarrow \fp :
  P \mapsto c \cdot \mathcal B(P).
 \]
 Then $c' \in \mc_1(n,q^e)^\perp$ and $\wt(c') \leq \wt(c)$.
\end{constr}

\begin{proof}
Take a line $l$ in $\pg(n,q^e)$.
Then we know that $\sett{\mathcal B(P)}{P \in l }$ is a partition of the points of $\mathcal B(l)$.
Therefore,
\[
 c' \cdot l
 = \sum_{P \in l} c'(P)
 = \sum_{P \in l} c \cdot \mathcal B(P)
 = \sum_{P' \in \cup_{P \in l} \mathcal B(P)} c(P')
 = c \cdot \mathcal B(l)
 = 0.
\]
The last equality holds because $\mathcal{B}(l)$ is a $(2e-1)$-space in $\pg(N,q)$ and $c \in \mc_{2e-1}(n,q)^\perp$.
If a point $P$ of $\pg(n,q^e)$ lies in $\supp(c')$, then $\mathcal B(P)$ must certainly contain a point of $\supp(c)$.
Since the spread $\mathcal S := \sett{\mathcal B(P)}{P \in G_0(n,q^e)}$ partitions the points of $\pg(N,q)$, $\supp(c')$ cannot contain more points than $\supp(c)$. 
\end{proof}

\begin{rmk}
If the code word $c$ in the above definition is a minimum weight code word of $\mc_{2e-1}(N,q)^\perp$, then it is embedded in an $((n-1)e+1)$-space $\pi$.
In that case, it's not hard to check that $\supp(c')$ are the points $P$ in $\pg(n,q^e)$, such that $\mathcal B (P)$ intersects $\pi$ in a single point and this point belongs to $\supp(c)$.
\end{rmk}

\section{Open problems}\label{SecOpenProblem}

A first open problem is solving the minimum weight problem of $\mc_1(n,q)^\perp$.
It would be interesting to investigate whether (all) minimum weight code words of $\mc_1(n,q)^\perp$, $n > 2$, come from Construction \ref{ConstrTruncCone}, and it would be delightful if the answer is positive.
In that case, the minimum weight problem is entirely reduced to $\mc_1(2,q)^\perp$, which remains an interesting case in itself.\\

It would also be nice if the characterisations for $\mc_{j,k}(n,q)$ can be improved beyond the bound $W(j,k,q)$, and if the minimum weight of $\mh_{j,k}(n,q)$ can be proven to be $2q^{k-j}\gauss k j$ for small values of $q$ as well.\\

Another important open problem remains a general dimension formula for these codes.
This dimension is known for $j = 0$ and is given by Hamada's formula \cite{hamada}.
Recall that $\mc_{j,k}(n,q)$ can be defined as the row span of the $p$-ary incidence matrix of $j$-spaces and $k$-spaces.
Hence, $\dim \left( \mc_{j,k}(n,q) \right)$ equals the rank of this matrix.
By duality, this matrix can also be seen as the transposed incidence matrix of $(n-k-1)$-spaces and $(n-j-1)$-spaces.
This implies that
\[
 \dim \left( \mc_{j,k}(n,q) \right)
 = \dim \left( \mc_{n-k-1,n-j-1}(n,q) \right).
\]
In particular, this means that Hamada's formula can be used to compute the dimension of $\mc_{j,n-1}(n,q)$, which equals $\dim ( \mc_{n-j-1}(n,q))$.\\

\section{Acknowledgements}

We would like to thank our promotors Leo Storme and Jan De Beule for their guidance and Maarten De Boeck for proofreading the manuscript.

We would also like to thank Andreas Bächle and Francesco Pavese to help us find the reference for Result \ref{ResOrderCollineation} and Bernhard Mühlherr for concocting an alternative proof for this result in an impressively short amount of time.

\bibliographystyle{alpha}
\bibliography{ref.bib}

Authors address:\\
Sam Adriaensen\\
Vrije Universiteit Brussel, Department of Mathematics\\
Pleinlaan $2$\\
$1050$ Brussels\\
BELGIUM\\
\texttt{e-mail: sam.adriaensen@vub.be}
	
\bigskip
Lins Denaux\\
Ghent University
Department of Mathematics: Analysis, Logic and Discrete Mathematics\\
Krijgslaan $281$ -- Building S$8$\\
$9000$ Ghent\\
BELGIUM\\
\texttt{e-mail: lins.denaux@ugent.be}\\

\end{document}